\documentclass[12pt]{amsart}

\usepackage{graphicx,amsfonts,amssymb,stmaryrd,amscd,amsmath,latexsym,amsbsy,hyperref,eufrak}

\usepackage{hyperref} 
\usepackage[T2A]{fontenc}
\usepackage[utf8x]{inputenc}
\usepackage[usenames]{color}
\usepackage{bbold}
\usepackage{mathtools}
\usepackage{color,amsthm}

\usepackage{tikz}

\usetikzlibrary{arrows,positioning,cd,calc,math}
\usetikzlibrary{matrix}
\usetikzlibrary{decorations.pathreplacing,decorations.markings}
\usetikzlibrary{shapes.geometric}

\theoremstyle{definition}

\newtheorem{theorem}{Theorem}[section]
\newtheorem{lemma}[theorem]{Lemma}
\newtheorem{proposition}[theorem]{Proposition}
\newtheorem{corollary}[theorem]{Corollary}
\theoremstyle{definition}
\newtheorem{definition}[theorem]{Definition}

\newtheorem{example}[theorem]{Example}
\newtheorem{exercise}[theorem]{Exercise}
\newtheorem{question}[theorem]{Question}

\newtheorem{remark}[theorem]{Remark}

\numberwithin{equation}{section}

\usepackage[left=2cm,right=2cm,top=2cm,bottom=2cm ,bindingoffset=0cm]{geometry}

\def\m{m}
\def\O{\mathcal{O}}
\def\id{\mathrm{id}}
\def\Rep{\mathrm{Rep}}
\def\Aut{\mathrm{Aut}\,}
\def\M1{1}
\def\V1{\mathbb{1}}
\def\Hom{\mathrm{Hom}\,}
\def\Vec{\mathrm{Vec}\,}
\def\Z{\mathcal{Z}}
\def\Aut{\mathrm{Aut}\,}
\def\hexd{\mathrm{hex}_1}
\def\hexdm{\mathrm{hex}_2}
\def\Ind{\mathrm{Ind}\,}
\def\YD{\mathcal{YD}}
\def\D{\mathcal{D}}
\def\op{\mathrm{op}}
\def\cop{\mathrm{cop}}
\def\Mod{\mathrm{mod}}
\def\R{\mathcal{R}}
\def\Fun{\mathrm{Fun}\,}
\def\End{\mathrm{End}\,}
\def\dim{\mathrm{dim}\,}

\usepackage{ulem}
\definecolor{mygreen}{rgb}{0.13, 0.55, 0.13}
\definecolor{PEred}{rgb}{0.81, 0.06, 0.13}
\definecolor{PEgreen}{rgb}{0.5, 0.5, 0.0}
\definecolor{PEblue}{rgb}{0.58, 0.34, 0.92}

\begin{document}

\title{A brief introduction to quantum groups} 
\author{Pavel Etingof and Mykola Semenyakin} 
\begin{abstract} These are lecture notes of a mini-course given by the first author in Moscow in July 2019, taken by the second author and then edited and expanded by the first author. They were also a basis of the lectures given by the first author at the CMSA Math Science Literature Lecture Series in May 2020. We attempt to give a bird's-eye view of basic aspects of the theory of quantum groups.
\end{abstract}

\maketitle

\centerline{\bf To Igor Moiseevich Krichever on his 70th birthday with admiration\footnote{This paper was finished on June 9, 2021. Sadly, Igor Moiseevich Krichever passed away on December 1, 2022.}}

\tableofcontents
\section{Introduction} 

The theory of quantum groups emerged in mid 1980s from the work of 
St. Petersburg mathematical physicists led by L. D. Faddeev, as well as papers  
by V. Drinfeld and M. Jimbo, who formulated its mathematical framework 
through the formalism of Hopf algebras created by algebraic topologists in the middle of the 20th century. The basics of this theory were outlined in Drinfeld's 1986 talk at the ICM, 
\cite{Dr}, which still remains an excellent reference on the subject in spite of the many developments 
that followed. 

Soon after their discovery quantum groups exploded into a subject of great interest to many mathematicians and revealed deep connections with many areas of mathematics and physics, including 
algebraic combinatorics, representation theory, algebraic geometry, integrable systems, algebraic topology, knot theory, category theory, harmonic analysis, quantum field theory, quantum computation, and others. The theory of quantum groups became a topic of thousands of papers and led to 
breakthroughs in several areas of mathematics. This included such milestone developments 
as the theory of Reshetikhin-Turaev invariants of 3-manifolds, the theory of canonical bases by G. Lusztig and crystal bases by M. Kashiwara, the theory of cluster algebras by S. Fomin and A. Zelevinsky, representation theory of quantum groups at roots of 1 (C. de Concini, V. Kac, G. Lusztig),  
the theory of Kazhdan-Lusztig equivalences and of classical and quantum KZ equations and their generalizations (V. Drinfeld, D. Kazhdan -- G. Lusztig, M. Finkelberg, V. Schechtman -- A. Varchenko, I. Frenkel -- N. Reshetikhin, G. Felder, E. Mukhin, V. Tarasov), the theory of quantum differential and difference equations for quiver varieties (M. Aganagic, D. Maulik, A. Okounkov, A. Smirnov and their collaborators), the theory of categorification and categorical Kac-Moody actions (L. Crane--I. Frenkel, J. Chuang--R. Rouquier, M. Khovanov--A. Lauda), the representation theory of Yangians and quantum affine algebras and its applications to solvable lattice models (V. Chari, B. Feigin, E. Frenkel, M. Jimbo, T. Miwa, H. Nakajima, N. Reshetikhin, V. Toledano-Laredo, M. Varagnolo, E. Vasserot), applications to geometric representation theory and the (quantum) geometric Langlands 
correspondence (S. Arkhipov, R. Bezrukavnikov, A. Braverman, M. Finkelberg, D. Gaitsgory, J. Lurie)
 and many other works and names that deserve to be mentioned; it is impossible to give even a remotely complete list here.   
  
The goal of these lectures were to give a bird's-eye view of the basics 
of the theory of quantum groups. Because this subject has become so vast, 
we are not able even to mention many of its important parts, and other parts are discussed 
very briefly. For example, to simplify the exposition, we mostly discuss the example of $\mathfrak{sl}_2$, leaving most of the details on the higher rank case outside 
the scope of this text. There are also not many proofs -- many of the simpler ones are offered as exercises. We nevertheless hope that these lectures can give the readers an impression of the spirit of this beautiful subject and encourage them to take a deeper look. 

The paper is organized as follows. In Section 2 we review the basic language of the theory of quantum groups -- Hopf algebras and monoidal categories. In Section 3 we discuss the formalism 
to talk about $R$-matrices --  quasitriangular Hopf algebras and braided tensor categories. 
In Section 4 we discuss the Drinfeld-Jimbo quantum groups $U_q(\mathfrak{g})$ and their quasiclassical limits; in particular, we introduce Poisson-Lie groups, Lie bialgebras and Manin triples. 
Finally, in Section 5 we discuss infinite dimensional quantum groups obtained by quantizing loop algebras (Yangians and quantum affine algebras), finishing with the theory of $q$-characters of E. Frenkel and N. Reshetikhin. 

The desirable prerequisites for the reader include basic abstract algebra and representation theory (especially finite dimensional representations of complex semisimple Lie algebras) and basics of category theory. The lectures contain many exercises (both in the text and at the end of Sections 2 and 3), which contain important material and are recommended for active reading. Also, at the beginning of each section we provide a list of relevant references. 

{\bf Acknowledgements.} We are very happy to dedicate this paper to the 70th birthday of Igor Moiseevich Krichever, who has been a mentor to both of us over the years. In particular, the first author has been greatly influenced by Igor Moiseevich ever since late 1980s when he was an undergraduate student in  Moscow. 

We are very grateful to the organizers and participants of the Skoltech Summer School on Mathematical Physics, July 1-12 2019, which led to this paper. 
The first author would also like to express special thanks to I. M. Krichever and A. Yu. Okounkov for their encouragement and creating an enjoyable and stimulating atmosphere. He is also grateful to Pablo Boixeda for leading problem sessions, without which this minicourse could not have succeeded, and for his comments on a draft of this paper. 

The work of the first author was partially supported by the NSF grant DMS - 1916120. The work of the second author is supported in part by the Government of Canada through the Department of Innovation, Science and Economic Development and by the Province of Ontario through the Ministry of Colleges and Universities.

\section{Hopf algebras and monoidal categories}

\subsection{The definition of a Hopf algebra} (\cite{CP}, Ch. 4; \cite{ES}, Ch. 8; \cite{K}, Ch. III).

\subsubsection{} Recall that in classical physics, we work with a {\bf space of states} $X$ and {\bf the algebra of classical observables} $A=\O(X)$ of (say, complex-valued) regular functions on $X$, which is a commutative, associative algebra. E.g., for a particle moving on the line, $X = \mathbb{R}^2$, $A=\Bbb C[x,p]$, and the coordinate $x$, the momentum $p$, and the energy $H=\frac{p^2}{2}+U(x)$ are examples of classical observables (elements of $A$). 

When we pass from classical physics to quantum physics, the algebra $A$ is deformed to a non-commutative (but still associative) {\bf algebra of quantum observables} $A_{\hbar}$ depending on a  quantization parameter $\hbar$ (the Planck constant). For instance, in the simplest example of $X=\Bbb R^2$, $A_\hbar=\Bbb C\langle x,p \rangle$, where $p=-i\hbar \partial_x$, so that we have 
the Heisenberg uncertainty relation $[p,x]=-i\hbar$.

\subsubsection{} Now consider the case when the classical space of states $X=G$ is a group. 
Then $G$ is equipped with an associative product, which has a unit and all elements are invertible:
\begin{equation}
\def\arraystretch{1.5}
\begin{array}{l}
\m: G  \times G \to G,\ m(x,y)=xy, \\ x(yz) = (xy)z,\\
\exists\, e\in G\ \forall g \in G:~ eg = ge = g, \\
\forall\, g\in G ~ \exists\, g^{-1}\in G: gg^{-1} =g^{-1}g= e.
\end{array}
\end{equation}
Thus the algebra $A=\O(G)$ of functions on a (say, finite) group has a natural structure of a {\bf coalgebra}. Namely, decomposing $\O(G\times G)$ as $\O(G)\otimes \O(G)$, one gets {\bf comultiplication}, or {\bf coproduct} on $A$ from the multiplication $m$ in $G$:
\begin{equation}
\Delta: A\to A\otimes A:~ \Delta (f)(x,y) = f(xy) = f_{(1)}(x)\otimes f_{(2)}(y). 
\end{equation}
Here we used {\bf Sweedler's notation} for the coproduct $\Delta(f) = f_{(1)}\otimes f_{(2)}$ where summation on the RHS is assumed, i.e., $f_{(1)}\otimes f_{(2)}$ is really $\sum_i f_{(1)}^i\otimes f_{(2)}^i$. 

\subsubsection{} Let us now discuss the properties of $\Delta$. First, it is clear that 

\begin{itemize} \item $\Delta$ is an algebra homomorphism.  \end{itemize}

The algebra $A$ also has a natural {\bf counit} and {\bf antipode} obtained from 
the unit and inversion in $G$: 
\begin{equation}
\def\arraystretch{1.5}
\begin{array}{l}
\varepsilon: A \to \mathbb{C}:~ \varepsilon(f) = f(e),\\
S: A \to A:~ S(f)(x) = f(x^{-1}).
\end{array}
\end{equation}

The following properties can be easily checked: 
\begin{itemize}
\item $\Delta$ is coassociative, i.e., $(\id\otimes \Delta)\circ\Delta =(\Delta\otimes \id )\circ\Delta $;
\item $(\varepsilon \otimes \id)\circ\Delta = (\id\otimes \varepsilon)\circ\Delta = \id$;
\item $\mu\circ (S\otimes \id)\circ \Delta(f)=\mu\circ (\id\otimes S)\circ \Delta(f)=\varepsilon(f)$,
where $\mu: A\otimes A\to A$ is the multiplication of functions. 
\end{itemize}

\subsubsection{} The notion of a {\bf quantum group} should therefore be obtained by dropping the commutativity of $A$. Namely, we make the following definition. 

\begin{definition} A {\bf quantum group} or {\bf Hopf algebra} is a unital associative algebra $A$ (not necessary commutative) which is equipped with $\Delta,\varepsilon,S$ and has the properties listed above.\footnote{The term ``quantum group" is usually used when $A$ is neither commutative nor cocommutative.} A {\bf Hopf subalgebra} of a Hopf algebra $H$ 
is a unital subalgebra $K\subset H$ such that $\Delta(K)\subset K\otimes K$ and $S(K)\subset K$. 
\end{definition}

Note that this definition makes sense over any field and even over a commutative ring. 

Usually it is also assumed that $S$ is invertible (we will do so below); this does not follow from the above axioms. 

Note that the coassocitivity of the coproduct implies that one can introduce Sweedler's notation 
for more than two components:  
$$
(\id\otimes \Delta)\circ\Delta(g) =(\Delta\otimes \id)\circ\Delta(g)= g_{(1)}\otimes g_{(2)}\otimes g_{(3)}.
$$
We will also use the notation $\Delta^{\op} = P\circ \Delta$ where $P$ is the permutation: $P(u\otimes v) = v \otimes u$.

\subsection{Properties and examples of Hopf algebras} (\cite{CP}, Ch. 4; \cite{ES}, Ch. 8; \cite{K}, Ch. III).

\subsubsection{} Basic properties of Hopf algebras are summarized in the following proposition. 

\begin{proposition} (i) If $H$ is a finite dimensional Hopf algebra then so is $H^*$, 
with the operations of $H^*$ being dual to the operations of $H$.\footnote{Here we should view the unit of $H$ as a linear map $\iota: \Bbb C\to H$ such that $\iota(1)=1$. Then the dual of $\varepsilon_H$ is $\iota_{H^*}$.} 

(ii) $\varepsilon$ is an algebra homomorphism.

(iii) $S$ is an algebra and coalgebra antihomomorphism, i.e., 
$S(xy)=S(y)S(x)$ and $\Delta(S(x))=(S\otimes S)(\Delta^{\rm op}(x))$. 

(iv) $\varepsilon$ and $S$ are uniquely determined by $\Delta$. 

(v) If $A$ is commutative or cocommutative (i.e., $\Delta=\Delta^{\rm op}$) then $S^2={\rm id}$
(even without the assumption that $S$ is invertible).
\end{proposition}

\subsubsection{} Here are some examples of Hopf algebras. 

\begin{example}
(i) $A=\O(G)$, $G$ is a finite group. Then $A$ is commutative. 
\item $A=\O(G)$ (the algebra of regular functions), $G$ is an affine algebraic group. Then $A$ is commutative. 

(ii) $A=\mathbb{C}G$ - the group algebra, $\Delta(g) = g\otimes g$, $S(g) = g^{-1},~ \varepsilon(g) = 1$, $g\in G$ (i.e., $g$ is a {\bf grouplike element}). In this case $A$ is cocommutative: $\Delta = \Delta^{\op}$.

(iii) $\mathfrak{g}$ is a Lie algebra, $A=U(\mathfrak{g})$ (the universal enveloping algebra), $\Delta(x) = x\otimes 1 + 1\otimes x$, $S(x)= -x$, $\varepsilon(x)=0$ for all $x\in \mathfrak{g}$ (i.e., $x$ is a {\bf primitive element}). Then $A$ is cocommutative.

(iv) {\it The quantum deformation of $U(\mathfrak{sl}_2)$} (a prototypical example of a quantum group). Let $q\in \Bbb C$, $q\ne 0,\pm 1$. Then the {\bf quantum group $U_q(\mathfrak{sl}_2)$} is defined by generators and relations as follows: 
\begin{equation}
U_q(\mathfrak{sl}_2) = \left\langle e,f,K^{\pm 1}\,|\, KeK^{-1} = q^2 e,~ K f K^{-1} = q^{-2} f,~ e f-f e= \frac{K-K^{-1}}{q-q^{-1}}  \right\rangle.
\end{equation}
\begin{equation}
\Delta(e) = e \otimes K + 1\otimes e,~ \Delta(f) = f \otimes 1 + K^{-1} \otimes f,~ \Delta(K) = K \otimes K,~ \varepsilon (e) = 0,~ \varepsilon(f) = 0,~ \varepsilon(K) = 1.
\end{equation}
The antipode could be uniquely found from the Hopf algebra axioms:
\begin{equation}
\mu\circ (S\otimes 1)\circ \Delta (e) = \varepsilon(e) ~\Rightarrow~ S(e)=-e K^{-1}.
\end{equation}
Similarly one obtains
\begin{equation}
S(f) = - K f,~~~ S(K) = K^{-1}.
\end{equation}
To go back to $U(\mathfrak{sl}_2)$, take $q = e^{\hbar/2}, K = q^{h}$ and send $\hbar \to 0$.
\end{example} 

\begin{remark} This example shows that $S^{2}\neq \id$ in general.
\end{remark}

\subsection{Monoidal categories} (\cite{EGNO}, Ch. 2; \cite{K}, Ch. XI).
\subsubsection{} For a group $G$ and a Lie algebra $\mathfrak{g}$ the representation categories $\Rep(G)$ and $\Rep (\mathfrak{g})$ are endowed with tensor products
\begin{equation}
\def\arraystretch{1.5}
\begin{array}{l}
\pi_V: G \to \Aut V,~\pi_W: G \to \Aut W ~ \mapsto ~ \pi_{V\otimes W}(g) = \pi_V(g)\otimes \pi_W(g) ~~~ \forall g \in G,\\
\pi_V: \mathfrak{g} \to \Aut V,~\pi_W: \mathfrak{g} \to \Aut W ~ \mapsto ~ \pi_{V\otimes W}(x) = \pi_V(x) \otimes \M1 + \M1 \otimes \pi_W(x) ~~~ \forall x \in \mathfrak{g}.
\end{array}
\end{equation}

Similarly, we can define the tensor product of representations of a Hopf algebra: 
\begin{equation}
\pi_{V\otimes W}(x) := (\pi_{V}\otimes \pi_{W})(\Delta (x)) = \pi_{V}(x_{(1)})\otimes \pi_{W}(x_{(2)}) ~~~ \forall x \in H.
\end{equation}
So one can regard the category $\mathcal{C}=\Rep(H)$ of representations of $H$ as a category equipped with a tensor product bifunctor
\begin{equation}
\otimes: \mathcal{C}\times \mathcal{C} \to \mathcal{C}:~ (X,Y)\mapsto X\otimes Y.
\end{equation}
This product also has a unit (the trivial 1-dimensional representation): 
\begin{equation}
\V1=\mathbb{C}:~ \pi_{\V1}(a)=\varepsilon(a) ,~ \V1\otimes X \cong X \otimes \V1 \cong X ~~~ \forall\, X \in \mathcal{C}.
\end{equation}
Finally, $\otimes$ is associative on isomorphism classes: 
\begin{equation}
(X\otimes Y)\otimes Z \cong X\otimes (Y\otimes Z). 
\end{equation} 
Thus $\mathcal C$ is {\bf a category with a unital tensor product associative  up to an isomorphism.} 

However, it turns out that this notion is not very useful; about such categories one can say very little, if anything at all. On the other hand, in natural examples a lot more structure is present, which is just a little bit harder to see. 

\subsubsection{} More precisely, a much better notion is obtained if, according to the general yoga of category theory, we don't just say simply that $(X\otimes Y)\otimes Z \cong X\otimes (Y\otimes Z)$, but make this isomorphism 
a part of the data and impose coherence conditions on this data.  
Namely, we should equip $\mathcal{C}$ with an {\bf associativity isomorphism} 
$$
\alpha_{XYZ}:~ (X\otimes Y)\otimes Z \xrightarrow{\sim} X\otimes (Y \otimes Z)
$$
functorial with respect to $X,Y,Z$, which satisfies the {\bf pentagon identity}
\begin{equation}
\begin{tikzpicture}
    \node [
      regular polygon,
      regular polygon sides=5,
      minimum width=75mm,
    ] (PG) {}
      (PG.corner 1) node (PG1) {$X\otimes (Y\otimes(Z\otimes T))$}
      (PG.corner 2) node (PG2) {$(X\otimes Y)\otimes (Z\otimes T)$}
      (PG.corner 3) node (PG3) {$((X\otimes Y)\otimes Z)\otimes T$}
      (PG.corner 4) node (PG4) {$(X\otimes(Y\otimes Z))\otimes T$}
      (PG.corner 5) node (PG5) {$X\otimes((Y\otimes Z)\otimes T)$}
    ;
    \foreach \S/\E in {
      5/1, 2/1, 3/2,3/4,4/5%
    } {
      \draw[->] (PG\S) -- (PG\E);
    }
\end{tikzpicture}
\end{equation}
(the diagram commutes) where the arrows are induced by $\alpha$. 

We should also require the existence of a {\bf unit object} $\V1$ with an isomorphism 
\begin{equation}
\iota:~ \V1 \otimes \V1 \cong \V1
\end{equation}
such that the functors $\V1\otimes$ and $\otimes \V1$ are autoequivalences of $\mathcal{C}$. 

\begin{definition} A category $\mathcal{C}$ with such structures and properties 
is called a {\bf monoidal category}. 
\end{definition}

\begin{remark}\label{coheren} The pentagon identity says that 
the two ways of identifying $((X\otimes Y)\otimes Z)\otimes T$ and $X\otimes (Y\otimes (Z\otimes T))$ using $\alpha$ give the same result. 
This allows us to identify any two brackettings of a tensor product of any number of factors, and the pentagon identity guarantees that any two ways of doing so using $\alpha$ give the same result.
The last statement is known as the {\bf Mac Lane coherence theorem for monoidal categories}. 
\end{remark} 

We see that for a Hopf algebra $H$, the category ${\rm Rep}(H)$ is a monoidal category, with 
$\alpha_{XYZ}$ being the natural isomorphism $(X\otimes Y)\otimes Z\to X\otimes (Y\otimes Z)$, sending $(x\otimes y)\otimes z$ to $x\otimes (y\otimes z)$. 

\begin{remark} A {\bf bialgebra} is defined in the same way as a Hopf algebra, but without the antipode (and antipode axioms). For ${\rm Rep}(H)$ to be a monoidal category, it suffices for $H$ to be a bialgebra. A basic example of a bialgebra is $\Bbb CG$ where $G$ is a monoid (not necessarily a group), with $\Delta(g)=g\otimes g$, $g\in G$. This explains the term ``monoidal category".  
\end{remark} 

\subsection{Duals and rigid categories} (\cite{EGNO}, Ch.2; \cite{K}, Ch. XIV). 
 
\subsubsection{} Let us now discuss duality for representations of Hopf algebras, which generalizes duality for group representations. The dual representations $X^*, \prescript{*}{}X \in \Rep(H)$ for $X\in \Rep(H)$ are both the dual vector space to $X$ with the actions defined as follows: 
\begin{equation}
\pi_{X^*}(a) = \pi_X(S(a))^*~\text{ - left dual},~~~ \pi_{\prescript{*}{}X}(a) = \pi_X(S^{-1}(a))^*~\text{ - right dual}.
\end{equation}
For the pair $X,X^*$ there is the {\bf evaluation morphism} 
\begin{equation} 
X^{*}\otimes X \to \V1
\end{equation} 
(the usual pairing). For finite dimensional representations there is also the {\bf coevaluation morphism} 
\begin{equation}
\V1 \to X \otimes X^*
\end{equation}
(the inverse of the usual pairing) and a pair of functorial isomorphisms
\begin{equation}
(\prescript{*}{}X)^{*} \cong X,~~~ \prescript{*}{}(X^{*})\cong X.
\end{equation}

\begin{remark} Since $S^2$ may not equal the identity, 
in general ${}^*X\ncong X^*$ and $X\ncong X^{**}$ as representations, 
even if $X$ is finite dimensional (in contrast with group representations). Thus 
one can define a countable family of (in general, pairwise distinct) iterated dual representations by
$X^{(2n+1)}=X^*$, $X^{(2n)}=X$ as vector spaces, and 
\begin{equation}
\pi_{X^{(2n+1)}}(a) = \pi_X(S^{2n+1}(a))^{*}, ~~
\pi_{X^{(2n)}}(a) = \pi_X(S^{2n}(a)),\ n\in \Bbb Z. 
\end{equation}
\end{remark}

\subsubsection{} We would now like to axiomatize this structure in the more general setting of monoidal categories. 

\begin{definition} An object $Y$ of a monoidal category $\mathcal{C}$ is a {\bf left dual} to $X$, denoted $Y=X^*$, if there exist {\bf evaluation and coevaluation} morphisms
\begin{equation}
\mathrm{ev}:~Y\otimes X \to \V1,~
\mathrm{coev}:~\V1 \to X\otimes Y,
\end{equation}
such that the following morphisms are the identities: 
\begin{equation}
\def\arraystretch{1.7}
\begin{array}{c}
\begin{CD}
X @ >{\mathrm{coev}\otimes 1}>> @ . (X\otimes Y)\otimes X@ >{\alpha_{XYX}}>> @ . X\otimes (Y\otimes X) @ >{1 \otimes \mathrm{ev}}>> @ . X,
\end{CD}
\\
\begin{CD}
Y @ >{1\otimes \mathrm{coev}}>> @ . Y\otimes (X\otimes Y)@ >{\alpha_{YXY}^{-1}}>> @ . (Y\otimes X)\otimes Y @ >{\mathrm{ev} \otimes 1}>> @ . Y.
\end{CD}
\end{array}
\end{equation}
\end{definition}

\begin{exercise} If $Y$ is a left dual to $X$ then we have a functorial isomorphism $$\Hom(Z,Y) \cong \Hom(Z\otimes X,\V1).$$
\end{exercise}
By the Yoneda lemma, this implies that the left dual, if exists, is unique up to a unique isomorphism preserving the evaluatioin and coevaluation morphisms. 

\begin{definition} An object $Z = \prescript{*}{}X$ is the {\bf right dual} to $X$ if $X\cong Z^*$.
\end{definition}

As the left dual, the right dual is unique up to a unique isomorphism preserving the evaluation and coevaluation morphisms if it exists. 

\begin{definition} An object $X$ is {\bf rigid} if it has both the left and the right dual. A category $\mathcal{C}$ is {\bf rigid} if all its objects are rigid.
\end{definition}

Thus, if $H$ is a Hopf algebra then the category ${\rm Rep}_f(H)$ of finite dimensional representations of $H$ is a rigid monoidal category. 

\subsubsection{}

\begin{example}\label{VecG} If $G$ is a finite group, $H=\O(G)$, then $\Rep_f(H) = \langle g\in G \,|\, \pi_{g}(f) = f(g) \rangle$ is spanned by 1-dimensional representations parametrized by $g\in G$. 
The tensor product is defined by $g\otimes h = gh$, and $g^*={}^*g=g^{-1}$. 
This gives $\Rep_f (H)$ the structure of a rigid monoidal category. We denote it by ${\rm Vec}(G)$
(as it is the category of $G$-graded vector spaces). Note that this category is well defined for any group $G$ (not necessarily finite). 

There is a ``purely multiplicative version" of this example: the rigid monoidal category 
$\mathcal{C}(G)$ whose objects are elements of $g$, $\Hom(g,h)=\emptyset$ 
if $g\ne h$ and $\Hom(g,g)=\Bbb C^\times $, with tensor product of objects $g\otimes h=gh$
and tensor product of morphisms $a\otimes b=ab$. In this case $(g\otimes h)\otimes k$ 
is simply equal to $g\otimes (h\otimes k)$, so we may take $\alpha_{g,h,k}=1$. 
In other words, $\mathcal C(G)$ is the subcategory of $\Vec(G)$ whose objects are all the 1-dimensional objects in  $\Vec(G)$ and morphisms are the same as in $\Vec(G)$, except that the zero morphisms are removed. We may say that the category ${\rm Vec}(G)$ is the {\bf linearization} of $\mathcal{C}(G)$. 
\end{example}

\begin{exercise} Consider a twisted version of this example. Let 
$$
\alpha_{g,h,k}:\,(g\otimes h)\otimes k=ghk \to g \otimes (h\otimes k)=ghk,
$$ 
i.e. 
$\alpha_{g,h,k}\in \Bbb C^\times $. Show that $\alpha$ satisfies the pentagon identity $\Leftrightarrow$ $\alpha$ is a 3-cocycle of the group $G$. Show that if so then this equips $\Rep_f (H)$ with another structure of a rigid monoidal category (with the same tensor product functor but different associativity isomorphism). We will denote this category $\Vec (G,\alpha)$. Similarly, we may define the twisted version $\mathcal C(G,\alpha)$ of 
$\mathcal C(G)$. 
\end{exercise} 

\subsection{Monoidal functors} (\cite{EGNO}, Ch. 2, \cite{K}, Ch. XI).

\subsubsection{} Let $\mathcal{C},\mathcal{D}$ be monoidal categories. 

\begin{definition} A functor $F: \mathcal{C} \to \mathcal{D}$ is a {\bf monoidal functor} if $F(\V1_\mathcal{C}) \cong \V1_\mathcal{D}$ and 
$F$ is equipped with a functorial (in $X,Y$) isomorphism 
\begin{equation}
 J_{X,Y}: ~ F(X)\otimes F(Y) \xrightarrow{\sim} F(X\otimes Y)
\end{equation}
which makes the diagram
\begin{equation}
\begin{CD}
(F(X)\otimes F(Y))\otimes F(Z) @ >{\alpha_{\mathcal{D}}}>> @ . F(X)\otimes (F(Y)\otimes F(Z))\\
@VV{J_{X,Y}\otimes {\rm id}_Z}V @ . @ VV{{\rm id}_X\otimes J_{Y,Z}}V\\
F(X \otimes Y)\otimes F(Z) @ . @ . F(X) \otimes F(Y\otimes Z)\\
@VV{J_{X\otimes Y,Z}}V @ . @ VV{J_{X, Y\otimes Z}}V\\
F((X \otimes Y)\otimes Z) @ >{F(\alpha_{\mathcal{C}})}>> @ . F(X \otimes (Y\otimes Z))
\end{CD}
\end{equation}
commutative. The functorial isomorphism $J$ is called the {\bf monoidal structure} of $F$. 
A monoidal functor is an {\bf equivalence of monoidal categories} if it is an equivalence of categories.
\end{definition}

The notion of monoidal equivalence is useful because monoidal categories that are monoidally equivalent may be viewed as ``the same for all practical purposes". 

\subsubsection{} The following exercise provides a classification of the categories 
$\Vec(G,\alpha)$ and $\mathcal C(G,\alpha)$ up to monoidal equivalence.  

\begin{exercise}\label{cohomo} If $\alpha,\beta$  are two 3-cocycles on $G$ then the identity functor
\begin{equation}
F={\rm Id}:~\Vec(G,\alpha)\to \Vec(G,\beta),~~~ g\xrightarrow{\sim} g
\end{equation}
is monoidal with $J_{g,h}\in \mathbb{C}^\times: g\otimes h \to g\otimes h$ iff $dJ = \alpha/\beta$, where $d$ is the differential in the standard complex of $G$ with coefficients in $\Bbb C^\times $.
In particular, $F$ admits a monoidal structure if and only if the cohomology classes of $\alpha$ 
and $\beta$ are the same. Thus the categories $\Vec(G,\alpha)$ 
up to monoidal equivalence are classified by $H^3(G,\Bbb C^\times)/{\rm Aut}(G)$. We also see that $\Vec(G,\alpha)$ is equivalent to $\Vec(G)$ iff $\alpha$ is trivial in $H^3(G,\Bbb C^\times )$. 

The same results apply to the categories $\mathcal C(G,\alpha)$. 
\end{exercise}

\subsection{Strict monoidal categories} (\cite{EGNO} 2.8, \cite{K}, XI.5).  

\subsubsection{} Associativity isomorphisms can make computations with monoidal categories somewhat cumbersome. However, the following shows that in such computations, we can,  in fact, ``forget" about associativity isomorphisms. 

\begin{definition} A monoidal category $\mathcal C$ is called {\bf strict} if $(X \otimes Y) \otimes Z = X \otimes (Y \otimes Z)$ with $\alpha_{XYZ}={\rm id}$ for all $X,Y,Z$, $\V1\otimes X=X\otimes \V1=X$ for all $X$, and $\iota={\rm id}$. 
\end{definition}

For example, the category $\mathcal C(G)$ defined in Example \ref{VecG} is strict, while $\mathcal C(G,\alpha)$ with $\alpha\ne 1$ is not. 

\begin{theorem}\label{stri} (Mac Lane's strictness theorem) Any monoidal category is equivalent to a strict monoidal category.
\end{theorem}

Note that Theorem \ref{stri} immediately implies Mac Lane's coherence theorem, see Remark \ref{coheren}. 

\subsubsection{} Theorem \ref{stri} may seem to contradict Exercise \ref{cohomo}: how can $\mathcal C(G,\alpha)$ with a cohomologically nontrivial $\alpha$ (which we assume to be {\bf skeletal}, i.e. there is just one object in each isomorphism class) be equivalent to a strict monoidal category? The resolution of this seeming contradiction comes from the subtle distinction between {\bf equivalent} monoidal categories and {\bf isomorphic} ones. Isomorphism (of small categories) is a stronger condition than equivalence: it means that we should not just have a bijection between {\bf isomorphism classes of objects} of the two categories, but must literally have a bijection between {\bf objects themselves.} Then $\mathcal C(G,\alpha)$ is {\bf isomorphic} to a strict category (as in the above exercise) iff $\alpha$ is trivial in $H^3(G,\Bbb C^\times )$, but it is always {\bf equivalent} to one by Mac Lane's strictness theorem. In other words, every category is equivalent to a skeletal one and every monoidal category is equivalent to a strict one, but {\bf not} every monoidal category is equivalent to one that it strict and skeletal at the same time. A counterexample is $\mathcal C(G,\alpha)$ for a cohomologically nontrivial $\alpha$.\footnote{The category of vector spaces should be viewed as non-strict, but is equivalent to a strict skeletal category, as is ${\rm Vec}(G)$ for any $G$. However, the category 
${\rm Rep}(G)$ of representations of a non-abelian finite group $G$ is not equivalent to a strict skeletal category.} 

So how can we explicitly construct a strict model for such a category? This turns out to be not completely trivial, which demonstrates the nontrivial nature of Mac Lane's strictness theorem (in spite of its simple proof). Namely, let us choose a group $\widetilde{G}$ 
with a surjective homomorphism $\phi: \widetilde{G}\to G$ such that $\phi^*\alpha$ is trivial in $H^3(\widetilde{G},\Bbb C^\times )$. This is possible, e.g. we can take $\widetilde{G}$ to be a free group, then its cohomology in degrees $\ge 2$ is zero. Then $\phi^*\alpha=dJ$ for some 2-cochain $J$ on $\widetilde{G}$. Let $\widetilde{\mathcal{C}}$ be the category whose objects are elements of 
$\widetilde{G}$, with $\Hom(g,h)=\Bbb C^\times$ if $g,h$ map to the same element of $G$ and $\emptyset$ otherwise. Define a strict monoidal structure on this category as follows: $g\otimes h=gh$ (but with tensor product of morphisms given by $a\otimes b=abJ_{g,h}$) and trivial associativity. 
Also let $\mathcal{D}$ be the same category with usual tensor product but associativity defined 
by $\phi^*\alpha$. Then $J$ defines a monoidal structure on the identity functor 
$\widetilde{\mathcal{C}}\to \mathcal{D}$. On the other hand, we have a natural monoidal equivalence 
$\widetilde{\mathcal{C}}\to \mathcal C(G,\alpha)$. Thus, $\mathcal C(G,\alpha)$ is monoidally equivalent (albeit not isomorphic) to the strict (but non-skeletal) monoidal category $\mathcal{D}$. 

\subsection{Problems} \label{prob1}
\begin{enumerate}
\item Let $H$ be a commutative or cocommutative Hopf algebra. Show that $S^2=\id$. 
\item Let $H$ be a Hopf algebra and $Y$ an $H$-module. Show that 
the $H$-module $Y\otimes H$ (with $H$ acting via $a\mapsto (\pi_Y\otimes \id)(\Delta(a))$) 
is isomorphic to $Y_{\rm space}\otimes H$ (i.e. $H$ acts only in the second factor, $a\mapsto 1\otimes a$). 

{\bf Hint.} Show that the linear map $\xi: Y_{\rm space}\otimes H\to Y\otimes H$ 
given by $y\otimes h\mapsto \Delta(h)(y\otimes 1)$ is a homomorphism of representations. 
Then use the antipode to construct the inverse $\xi^{-1}$ and thus show that $\xi$ is an isomorphism. 
 
\item Show that $\forall~n\geq 1~\exists$ a Hopf algebra $T_n(q)=\langle g,x \rangle$ with defining relations 
$$
g^n=1,~x^n=0,~gx = q xg
$$ 
and coproduct defined by 
$$
\Delta(g) = g \otimes g,\ \Delta(x)=x\otimes \M1 +g\otimes x,
$$ 
where $q$ is a primitive $n$-th root of unity, and that $\dim T_n(q)=n^2$ (it is called the {\bf Taft Hopf Algebra}). Compute the antipode $S$ of $T_n$.\footnote{As was shown in \cite{Ng}, if $n$ is a prime 
then $T_n(q)$ are the only Hopf algebras of dimension $n^2$ over $\Bbb C$ which are not commutative and cocommutative, i.e., not isomorphic to the group algebra of an abelian group.}
\item Show that $\forall~n\geq 1~\exists$ a Hopf algebra $H_n=\langle g,x_1,...,x_n \rangle$ 
with defining relations 
$$
g^2=1,~g x_i = -x_i g,~ x_i x_j + x_j x_i = 0
$$ and coproduct defined by 
$$
\Delta (g) = g \otimes g,\ \Delta (x_i)=x_i\otimes \M1 +g\otimes x_i,
$$ 
of dimension $2^{n+1}$ (it is called the {\bf Nichols Hopf algebra}). Compute the antipode $S$ of $H_n$. Show that $H_2=T_2(-1)$ (it is called the {\bf Sweedler 4-dimensional Hopf algebra}). 
\item Let $H$ be a Hopf algebra over $\mathbb{C}$, $G$ - a group acting by automorphisms of $H$. Show that the semidirect product algebra $\mathbb{C}G\ltimes H $ has a unique Hopf algebra structure in which $\mathbb{C}G$ and $H$ are Hopf subalgebras.
\item Show that every cocommutative finite dimensional Hopf algebra over $\Bbb C$ 
has the form $\Bbb C G$, where $G$ is a finite group. In particular, it is semisimple. 
(Problems 3,4 show that the latter is not true in the non-cocommutative case). 
\item Show that the Hopf algebra $U_q(\mathfrak{sl}_2)$ is well defined by generators $e,f,K^{\pm 1}$ and defining relations
\begin{equation}
KeK^{-1} = q^2 e,~ K f K^{-1} = q^{-2} f,~ [e,f] = \frac{K-K^{-1}}{q-q^{-1}},
\end{equation}
with coproduct defined by
\begin{equation}
\Delta(e) = e \otimes K + 1\otimes e,~ \Delta(f) = f \otimes 1 + K^{-1} \otimes f,~ \Delta(K) = K \otimes K. 
\end{equation}
\item Show that a grouplike element $g$ of a Hopf algebra $H$ defines a continuous 
(in the weak topology) 1-dimensional representation of $H^*$, and vice versa. 
\item Let $g,h$ be grouplike elements of a Hopf algebra $H$. An element $x\in H$ 
is called {\bf $(g,h)$-skewprimitive} if $\Delta(x)=g\otimes x+x\otimes h$. 
For example, a primitive element is $(1,1)$-skewprimitive, and the elements $e,f$ 
in $U_q(\mathfrak{sl}_2)$ are $(1,K)$- and $(K^{-1},1)$-skewprimitive, respectively. 
The element $\lambda(g-h)$ for $\lambda\in \Bbb C$ is an obvious example of a $(g,h)$-skewprimitive 
element. Such skewprimitive elements are called {\bf trivial}. Let ${\rm Prim}_{g,h}(H)$ 
be the space of $(g,h)$-skewpimitive elements in $H$ modulo the trivial ones. 
Show that ${\rm Prim}_{g,h}(H)\cong {\rm Ext}^1_{H^*}(h,g)$, where 
$g,h$ are viewed as continuous $H^*$-modules as in the previous problem. 
\item 
Show that the element 
\begin{equation}
C := f e + \dfrac{qK + q^{-1} K^{-1}-2}{(q-q^{-1})^2} 
\end{equation}
is central in $U_q(\mathfrak{sl}_2)$, and if $q$ is not a root of unity then the center 
of $U_q(\mathfrak{sl}_2)$ is generated by $C$. The element $C$ is called the {\bf quantum Casimir element.} 
\item Let $m$ be an integer. Show that $\exists !$ module $M_m$ over $U_q(\mathfrak{sl}_2)$ (called the {\bf Verma module} with highest weight $m$) with basis $v,fv,f^2 v,...$ such that $Kv = q^m v$. Compute the action of $K$ and $e$ on $M_m$. Use $M_m$ to prove the PBW theorem for $U_q(\mathfrak{sl}_2)$: the elements $K^{a}f^b e^c$ for $a\in \mathbb{Z},~b,c\in \mathbb{Z}_+$ form a basis in $U_q(\mathfrak{sl}_2)$.
\item Show that if $q$ is not a root of $1$ then the module $M_m$ 
is irreducible if $m<0$, but if $m\ge 0$ then $M_m$ contains $M_{-m-2}$ as a submodule, 
and $L_m:=M_m/M_{-m-2}$ is an irreducible $U_q(\mathfrak{sl}_2)$-module of dimension $m+1$.
\item \label{typeI} Let us say that a finite-dimensional $U_q(\mathfrak{sl}_2)$-module $L$ (for $q$ not a root of $1$) is of {\bf type I} if the eigenvalues of $K$ on $L$ are integer powers of $q$. 
Show that every irreducible type I module is $L_m$ for some $m$, and any type I module 
is completely reducible (for the last statement use the quantum Casimir $C$). This shows
that the representation theory of $U_q(\mathfrak{sl}_2)$ away from roots of unity 
is completely parallel to the representation theory of the Lie algebra 
$\mathfrak{sl}_2$ (at least if we restrict ourselves to type I representations). 
\item Let $\psi$ be the 1-dimensional representation of $U_q(\mathfrak{sl}_2)$ 
given by $\psi(K)=-1$, $\psi(e)=\psi(f)=0$ (it is not of type I). Show that any 
finite dimensional representation of $U_q(\mathfrak{sl}_2)$ (for $q$ not a root of $1$) 
is a direct sum of representations $L_m$ and $L_m\otimes \psi$. 
\end{enumerate}

\section{Braided monoidal categories and quasitriangular Hopf algebras}\label{qtr1}

\subsection{Braided monoidal categories and the Drinfeld center} (\cite{EGNO}, Ch. 8, \cite{K}, Ch. XIII). 

\subsubsection{} If $H$ is a Hopf algebra and $X,Y\in \,\Rep(H)$ then $X\otimes Y \ncong Y\otimes X$ in general, since $\Delta$ may not be cocommutative. However, even if $\Delta$ is not cocommutative, sometimes $X\otimes Y  \cong Y\otimes X$.

\begin{example} For $H=U_q(\mathfrak{sl}_2)$ where $q$ is not a root of unity 
define {\bf the universal $R$-matrix}
\begin{equation}\label{uniR}
\R = q^{\frac{h\otimes h}{2}}~\sum\limits_{k=0}^{\infty} q^{\frac{k(k-1)}{2}}\dfrac{(q-q^{-1})^k}{[k]_q !}e^{k}\otimes f^{k},
\end{equation}
where $[k]_q:=\frac{q^k-q^{-k}}{q-q^{-1}}$ and $[k]_q!=[1]_q...[k]_q$. 
This is an infinite series, but it makes sense as an operator on $X\otimes Y$ for any finite 
dimensional representations $X,Y$, because the sum terminates. Here $q^{\frac{h\otimes h}{2}}(x\otimes y):=q^{\frac{\lambda\mu}{2}}x\otimes y$ if $x$ has weight $\lambda$ and $y$ has weight $\mu$ (i.e., $Kx=q^\lambda x,Ky=q^\mu y$). 

\begin{theorem}\label{uniR1} (Drinfeld) The operator $c = P\circ \R$ defines an isomorphism of representations $c:~X\otimes Y \to Y\otimes X$. In other words, we have $\mathcal R\Delta(a)=\Delta^{\rm op}(a)\mathcal R$ on $X\otimes Y$ for $a\in H$. 
\end{theorem}

A motivation for formula \eqref{uniR} and a sketch of proof of Theorem \ref{uniR1} will be given in Subsection \ref{qdsl}.
\end{example}

\subsubsection{} This motivates the following definition. 

\begin{definition} A {\bf braided monoidal category} is a monoidal category $\mathcal{C}$ endowed with a functorial isomorphism $c:~\otimes \to \otimes^{\op},~c_{X,Y}: X\otimes Y \to Y \otimes X$ which satisfies the {\bf hexagon axioms} $\hexd$ and $\hexdm$:
\begin{equation}
\begin{array}{ccc}
X\otimes \underline{Y \otimes Z} & \rightarrow & Y\otimes X \otimes Z\\
 & \searrow & \downarrow\\ 
 &  & \underline{Y\otimes Z} \otimes X
\end{array}
\hspace{2cm}
\begin{array}{ccc}
\underline{X\otimes Y} \otimes Z & \rightarrow & X\otimes Z \otimes Y\\
 & \searrow & \downarrow\\ 
 &  & Z\otimes \underline{X \otimes Y}
\end{array}
\end{equation}
\end{definition}
Namely, $\hexd$ says that swapping $X$ with $Y\otimes Z$ in one step and in two steps gives the same result, and $\hexdm$ says the same about swapping $X\otimes Y$ with $Z$. 
Here we suppress associativity isomorphisms (a permissible simplification thanks to the Mac Lane coherence theorem), which is why our hexagon diagrams look more like triangles (there are three extra edges for the associativity isomorphisms in each of them that we have suppressed). 

\begin{proposition}\label{bract} If $\mathcal{C}$ is a braided monoidal category and $V\in \mathcal{C}$ then $ V^{\otimes n}$ has a natural action of the braid group
\begin{equation}
B_n = \pi_1((\mathbb{C}^n\backslash \text{diagonals})/S_n)
=
\langle s_1,...,s_{n-1}~|~ s_i s_j = s_j s_i~\text{if}~ |i-j|>1,~s_i s_{i+1} s_i = s_{i+1} s_i s_{i+1}\rangle
\end{equation} 
given by
\begin{equation}
\rho:~B_n \to \Aut (V^{\otimes n}),~~~ \rho(s_i) = c_{i,i+1} := (c_{V,V})_{i,i+1}. 
\end{equation}
\end{proposition}

\begin{exercise} Prove Proposition \ref{bract}. 
\end{exercise}

\begin{remark} If $\mathcal{C}$ is a braided monoidal category with braiding $c$ then the functorial morphism 
$c^{-1}$ defined by 
$(c^{-1})_{XY}:=(c_{YX})^{-1}$ is also a braiding on $\mathcal{C}$, called the {\bf inverse braiding} to $c$. The braiding $c$ is called {\bf symmetric} if $c=c^{-1}$, i.e., $c_{XY}\circ c_{YX}=\id \otimes \id$ for all $X,Y\in \mathcal{C}$.  
\end{remark}

\subsubsection{} An important example of a braided monoidal category is the {\bf Drinfeld center} of a monoidal category, defined by V. Drinfeld (unpublished) and independently by S. Majid (\cite{Ma}) and Joyal-Street (\cite{JS}).

\begin{definition} The {\bf Drinfeld center} $\Z(\mathcal{C})$ of a monoidal category $\mathcal{C}$ is the category of pairs $(Y,\varphi)$ where $Y\in \mathcal{C}$ and $\varphi:~Y\otimes - \to - \otimes Y$ is a functorial isomorphism given by the collection of isomorphisms $\varphi_X:~ Y\otimes X \xrightarrow{\sim} X\otimes Y, ~~ \forall\, X\in \mathcal{C}$, satisfying the following commutative diagram:  
\begin{equation}
\label{diag:hexagon}
\begin{CD}
Y \otimes (X_1\otimes X_2) @ >{\varphi_{X_1\otimes X_2}}>> @ .  (X_1\otimes X_2)\otimes Y\\
@VV{\alpha_{YX_1 X_2}^{-1}}V @ . @ AA{\alpha_{X_1 X_2 Y}^{-1}}A\\
(Y \otimes X_1)\otimes X_2 @ . @ . X_1 \otimes (X_2\otimes Y)\\
@VV{\varphi_{X_1}\otimes \id}V @ . @ AA{\id\otimes \varphi_{X_2}}A\\
(X_1 \otimes Y)\otimes X_2 @ >{\alpha_{X_1 Y X_2}}>> @ . X_1\otimes (Y \otimes X_2)
\end{CD}
\end{equation}
Morphisms of such pairs are morphisms in $\mathcal C$ which preserve $\varphi$. 
\end{definition}

The category $\mathcal Z(\mathcal C)$ has a natural monoidal structure defined by
$$
(Y,\varphi^Y)\otimes (Z,\varphi^Z)=(Y\otimes Z,\varphi^{Y\otimes Z}),
$$ 
 where 
$$
\varphi^{Y\otimes Z}_X=(\varphi^Y_X\otimes {\rm id}_Z)\circ ({\rm id}_{Y}\otimes \varphi^Z_X).
$$
Note that we have a natural monoidal forgetful functor $Z(\mathcal{C}) \to \mathcal{C},~ (Y,\varphi) \mapsto Y$.

\subsubsection{} If $V,W\in \Z(\mathcal{C})$ then there are two ways to identify $V\otimes W$ and $W\otimes V$, namely $\varphi^V_W$ and $(\varphi^W_V)^{-1}$.

\begin{proposition} $\mathcal Z(\mathcal{C})$ is a braided monoidal category with braiding $c_{VW}:=\varphi^V_W$.
\end{proposition} 

Thus $(\varphi^W_V)^{-1}=(c^{-1})_{VW}$ is the inverse braiding.

Note that if $\mathcal{C}$ is a braided monoidal category then we have a natural 
braided monoidal functor $\mathcal{C}\to \Z(\mathcal C)$ given by $X\mapsto (X,c_{X,-})$. 
Moreover, the composition $\mathcal{C}\to \Z(\mathcal{C}) \to \mathcal{C}$ is the identity. Thus every braided monoidal category is a full subcategory of its Drinfeld center.

\subsection{Yetter-Drinfeld modules and the quantum double} (\cite{EGNO}, 7.13 - 7.15, \cite{K}, XIII.4, XIII.5). 

\subsubsection{} Let $H$ be a Hopf algebra and $\Rep(H)$ be the category of finite dimensional $H$-modules. We would like to describe the Drinfeld center $\Z(\Rep(H))$ explicitly.  

Let $(Y,\varphi)\in \Z(\Rep(H))$. Then in particular we have a morphism 
\begin{equation}
\varphi_H: \Ind_{\mathbb{C}}^{H} Y\cong Y\otimes H \to H\otimes Y.
\end{equation} 
But we know that $Y\otimes H\cong Y_{\rm space}\otimes H$ (Subsection \ref{prob1}, Problem 2). 
Thus by Frobenius reciprocity, this morphism corresponds to a linear map
$$
\tau = \overline{\varphi}_H: Y \to H\otimes Y. 
$$
\begin{exercise} (i) Show that $\tau$ is an {\bf $H$-comodule} structure on $Y$ 
(i.e., it defines an $H^*$-module structure). 

(ii) Show that $\tau$ satisfies the following compatibility condition with the $H$-action on $Y$: 
\begin{equation}\label{compa}
\tau(ay) = a_{(1)} y_{(1)} S(a_{(3)}) \otimes a_{(2)}y_{(2)},
\end{equation}
where $y\in Y, ~\tau(y) = y_{(1)}\otimes y_{(2)}$ and $a\in H,~ (\id\otimes \Delta)\circ \Delta(a) = a_{(1)} \otimes a_{(2)} \otimes a_{(3)}$. 
\end{exercise} 

\begin{definition} A {\bf Yetter-Drinfeld module} over $H$ is an $H$-module $Y$ which is also an $H$- comodule with the comodule structure $\tau$ satisfying the compatibility condition \eqref{compa}. 
\end{definition}

\begin{proposition}\label{equi} The category $\Z(\Rep(H))$ is equivalent to the category  $\YD(H)$ of finite dimensional Yetter-Drinfeld modules over $H$.
\end{proposition}

\begin{exercise} Prove Proposition \ref{equi}.
\end{exercise} 

\subsubsection{} Now assume that $H$ is a finite dimensional Hopf algebra. In this case it turns out that  there exists a finite dimensional Hopf algebra $\D(H)$ called the {\bf quantum double of $H$} such that
\begin{equation}
\Z(\Rep(H))\cong \YD(H) \cong \Rep (\D(H)). 
\end{equation} 

Let us describe this Hopf algebra explicitly. 
Let $H^{\cop}=(H,\Delta^{\op},S^{-1})$ be the opposite Hopf algebra to $H$. 

\begin{exercise} Show that the actions of 
$H$ and $H^{*\cop}$ on $Y\in \YD(H)$ combine into an action on $Y$ of the free product 
$H\ast H^{*\cop}$ modulo the compatibility relations
\begin{equation}
ba = (a_{(1)},b_{(1)})(a_{(3)},b_{(3)})a_{(2)}S^{-1}(b_{(2)}),\ a\in H, b\in H^{*\cop}. 
\end{equation}
where $(x,y)$ denotes the pairing of $x\in H$ and $y\in H^*$.  

(ii) Show that the multiplication map 
$$
H\otimes H^{*\cop} \to (H\ast H^{*\cop})/(\text{compatibility relations}) 
$$
is a vector space isomorphism, and under the identification defined by this isomorphism, the 
multiplication in this algebra takes the form
\begin{equation}\label{multfor}
(a'\otimes b)(a\otimes b')= (a_{(1)},b_{(1)})(a_{(3)},b_{(3)})a' a_{(2)}\otimes S^{-1}(b_{(2)}) b', \ a,a'\in H, b,b'\in H^*.
\end{equation} 
\end{exercise} 

\begin{theorem} (Drinfeld) This defines a Hopf algebra structure on $\D(H):=H\otimes H^{*\cop}$ (with coproduct obtained by tensoring the coproducts of the factors), and $\YD(H)\cong \D(H)-\Mod$ as monoidal categories. 
\end{theorem}

\begin{definition} $\D(H)$ is called the {\bf Drinfeld double} or {\bf quantum double} of $H$.
\end{definition}

\subsubsection{} The properties of the quantum double are summarized in the following theorem. 

\begin{theorem}
(i) $H,H^{*\cop}\subset \D(H)$ are Hopf subalgebras.

(ii) $S_{\D(H)} = S_{H}\otimes S_{H^{*\cop}}$.

(iii) The braiding $c$ of $\Rep (\D(H))$ is given by $c_{XY}=P\circ \R|_{X\otimes Y}$, where the element 
$$
\R\in H\otimes H^* \subset \D(H)\otimes \D(H)
$$ 
is  defined by the formula
\begin{equation}
\R = \sum_i a_i \otimes a_i^*,
\end{equation}
where $\lbrace a_i\rbrace$ is a  basis of $H$ and $\lbrace a_i^*\rbrace$ is the dual basis of $H^*$.

(iv) $\R$ satisfies the {\bf quantum Yang-Baxter equation} (QYBE)
$$
\R^{12}\R^{13}\R^{23} = \R^{23}\R^{13}\R^{12}\in \D(H)^{\otimes 3},
$$
where if $\R:=\sum_i a_i\otimes b_i$ then $\R^{12}:=\R \otimes 1=\sum_i a_i\otimes b_i\otimes 1$, $\R^{23}:=1\otimes \R=\sum_i 1\otimes a_i\otimes b_i$, and $\R^{13}:=\sum_i a_i\otimes 1\otimes b_i$. 

(v) The multiplication formula \eqref{multfor} for $\D(H)$ is the unique one for which $\R$ satisfies QYBE.

(vi) $\R^{-1} = (S\otimes \id)\R$. 
\end{theorem} 

\subsection {Group actions on monoidal categories, equivariantization, and the Drinfeld center of ${\rm Vec}(G)$} (\cite{EGNO}, 2.7, 4.15). 

\subsubsection{}
\begin{example} Let $G$ be a finite group, and consider the category 
$\Vec (G)$ of $G$-graded complex vector spaces. What is its Drinfeld center?

We have $\Vec(G)=\Rep(H)$, where $H=\Fun(G,\Bbb C)$.
Thus 
\begin{equation}
\D(H) = H\otimes H^{*\cop} = \Fun(G,\Bbb C)\rtimes \mathbb{C} G ~\text{ (G acts on itself by conjugation)},~\R = \sum_{g\in G} \delta_g \otimes g,
\end{equation}
where $\delta_g$ is a function equal to $1$ on $g$, and to $0$ on all the other elements, i.e., the semidirect product Hopf algebra from Problem 5 of Subsection \ref{prob1}. Thus we have an equivalence 
\begin{equation}
\Z(\Vec (G)) \cong \text{\{G-equivariant, G-graded vector spaces\}}.
\end{equation}
The simple objects in this category are given by the Mackey theory -- they correspond to pairs $(C,\rho)$ where $C$ is a conjugacy class in $G$ and $\rho$ is an irreducible representation of the centralizer $Z(g_C)$, $g_C\in C$.
\end{example}

\subsubsection{} This brings us to the notion of {\bf group actions on categories}.  
\begin{definition} We say that a group $G$ {\bf acts on the category} $\mathcal{C}$ if
to each element $g\in G$ there is attached an autoequivalence $F_g:~\mathcal{C} \to \mathcal{C}$, and to every 
$g,h\in G$ there is attached a functorial isomorphism $\beta_{g,h}: F_g \circ F_h \xrightarrow{\sim} F_{gh}$ satisfying the consistency condition
\begin{equation}
\beta_{gh,k}\circ \beta_{g,h} = \beta_{g,hk}\circ \beta_{h,k}.
\end{equation}
\end{definition}

Note that since $F_1\circ F_1\cong F_1$, we have $F_1\cong \id$, hence 
$F_{g^{-1}}\cong F_g^{-1}$ for all $g\in G$.  

Let $(F_g,\beta_{g,h})$ and $(F_g',\beta_{g,h}')$ be two actions of $G$ on $\mathcal{C}$. 
An {\bf isomorphism} between them is a collection of functorial isomorphisms 
$\gamma_g: F_g\to F_g'$ that transform $\beta_{g,h}$ to $\beta_{g,h}'$. 

We may also define the notion of an action of a group on a {\bf monoidal} category $\mathcal{C}$. 
The definition is the same, except the functors  $F_g$ are required to be monoidal and 
the isomorphisms $\beta_{g,h}$ are required to be isomorphisms of monoidal functors. 
Finally, in the notion of isomorphism of actions, the morphisms $\gamma_g$ 
should be isomorphisms of monoidal functors. 

\begin{example} Let us classify actions of $G$ on the category $\Vec$ of vector spaces up to an isomorphism. First consider actions on $\Vec$ as an ordinary category. Since $F_g$ 
is an equivalence for all $g$, we may assume that 
\begin{equation}
\forall~g\in G\quad F_{g}=\id.
\end{equation}
Then 
\begin{equation}
\beta_{g,h}:~\id \xrightarrow{\sim} \id 
\end{equation} 
belongs to ${\rm Aut}(\id_\Vec)=\mathbb{C}^\times$. Moreover, the consistency condition 
is equivalent to saying that $\beta$ is a 2-cocycle: 
$\beta \in Z^2(G,\mathbb{C}^\times)$, and $\beta$ and $\beta'$ define isomorphic actions if and only if they differ by a coboundary: $\beta/\beta' = d\gamma$. 
So equivalence classes of actions are classified by $H^2(G,\mathbb{C}^\times)$. 

On the other hand, if $G$ acts on $\Vec$ as on a {\it monoidal} category then 
$\beta_{g,h}\in \Aut_{\otimes}(\id_\Vec)=1$. 
So there is only one (trivial) monoidal action of $G$ on $\Vec$. 
\end{example}

\subsubsection{}
Let $(F_g,\beta_{g,h})$ be an action of a group $G$ on a category $\mathcal{C}$. 
  
 \begin{definition} A $G${\bf -equivariant object} $X$ in $\mathcal{C}$ is an object equipped with isomorphisms $\gamma_g:~X\xrightarrow{\sim}F_g(X)$, $g\in G$ such that the diagram
\begin{equation}
\begin{CD}
X @ >{\gamma_g}>> @ . F_g(X)\\
@VV{\gamma_{gh}}V @ . @ VV{F_g(\gamma_h)}V\\
F_{gh}(X) @ <<{\beta_{g,h}}< @ . F_g F_h(X)
\end{CD}
\end{equation}
is commutative for any $g,h\in G$. 
\end{definition}

\begin{exercise} Suppose $G$ acts on $\Vec$ as a usual category via $\beta\in Z^2(G,\mathbb{C}^\times)$. Then $G$-equivariant objects are vector spaces $X$ with
\begin{equation}
\gamma_g:\,X \to X: \ \forall~g\in G,~~~ \gamma_{gh} = \beta_{g,h}\circ F_g(\gamma_h)\circ \gamma_g
\end{equation}
with $\beta_{g,h}\in \mathbb{C}^\times$. Thus $\gamma$ is just a {\bf projective representation} of $G$ on $X$ with Schur multiplier $\beta$. 

On the other hand, for a monoidal action $\beta=1$, so
\begin{equation}
\{\text{G-equivariant objects in } \Vec\}=\Rep(G).
\end{equation}
\end{exercise}

If $G$ acts on $\mathcal{C}$ then the category of $G$-equivariant objects in $\mathcal{C}$ is denoted by $\mathcal{C}^G$ and called the {\bf equivariantization} of $\mathcal{C}$ by $G$. If 
$\mathcal{C}$ is monoidal and $G$ acts monoidally on $\mathcal{C}$ then the equivariantization
$\mathcal{C}^G$ is a monoidal category. 

\subsubsection{}
\begin{exercise}\label{shortex} (i) Let a group $G$ act on a Hopf algebra $H$ by Hopf algebra automorphisms. 
Then we can define the semidirect product Hopf algebra $\mathbb{C}G\ltimes H$ as in Problem 5 of Subsection \ref{prob1}. Show that
\begin{equation}
\Rep (\mathbb{C}G\ltimes H) \cong (\Rep(H))^{G}.
\end{equation} 

(ii) Let $G$ be a group, ${\rm Aut}(G)$ 
its automorphism group, ${\rm Inn}(G)$ the subgroup of inner automorphisms
and ${\rm Out}(G)={\rm Aut}(G)/{\rm Inn}(G)$ 
the group of classes of automorphisms of $G$. 
Show that ${\rm Out}(G)$ acts naturally on the monoidal category $\Rep(G)$. 

(iii) Let 
$$
1\to K\to G\to L\to 1
$$
be a short exact sequence of groups. Show that the group $L$ acts naturally on the category 
$\Rep(K)$, and $\Rep(K)^L\cong \Rep(G)$ as monoidal categories. 
\end{exercise}

\begin{example} Thus we see that the Drinfeld center of $\Vec(G)$ can be described as follows:  
\begin{equation}
\Z(\Vec(G)) = \Rep \left(\mathbb{C}G\ltimes \Fun G\right) = \left(\Vec (G)\right)^{G}.
\end{equation}
\end{example}  

\begin{example} (\cite{N}) Let $\omega\in Z^3(G,\Bbb C^\times)$ be a 3-cocycle, 
and $\omega^g$ be the result of conjugation by $g\in G$ acting on $\omega$. 
Then there are natural functions 
$\mu_g(x,y)$, $\gamma_{g,h}(x)$ 
such that 
$$
d\mu_g(x,y,z)=\frac{\omega^g(x,y,z)}{\omega(x,y,z)},
$$
$$
\frac{\gamma_{g,h}(x)\gamma_{g,h}(y)}{\gamma_{g,h}(xy)}=\frac{\mu_g(hxh^{-1},hyh^{-1})\mu_h(x,y)}{\mu_{gh}(x,y)}.
$$
(see \cite{N}, Lemma 6.3). This allows us to define a $G$-action on $\Vec(G,\omega)$ 
such that $F_g(x)=gxg^{-1}$, the tensor structure on $F_g$ is defined by $\mu_g$, and 
$\beta_{g,h}|_x$ is given by $\gamma_{g,h}(x)$. Moreover, 
the Drinfeld center $\mathcal{Z}(\Vec(G,\omega))$ 
is the equivariantization $\Vec(G,\omega)^G$ of $\Vec(G,\omega)$ with respect to this action. 
\end{example}

\subsection{Tannakian reconstruction and quasitriangular Hopf algebras} (\cite{EGNO}, 5.1-5.4).

\subsubsection{} 
Let $A$ be an associative unital algebra, $\mathcal{C}=A-\Mod$, $F:A-\Mod \to \Vec$ the forgetful functor.

\begin{exercise} Show that $\End F \cong A$.
\end{exercise}

More precisely for any $a\in A$ we have a family of morphisms $a_{Z}: F(Z)\to F(Z),~Z\in \mathcal{C}$ such that for any $X, Y\in \mathcal{C}$ and $\beta: X \to Y$ the following diagram is commutative:
\begin{equation}
\begin{CD}
F(X) @ >{a_X}>> @ . F(X)\\
@VV{F(\beta)}V @ . @ VV{F(\beta)}V\\
F(Y) @ >{a_Y}>> @ . F(Y)
\end{CD}
\end{equation}
Now let $A=H$ be a bialgebra. As an associative algebra $\End F \cong H$.
But how can we recover the coproduct of $H$ from $F$?
Note that $\End F\otimes \End F\cong \End(F\otimes F)$, 
where $F\otimes F: \mathcal C\times \mathcal C\to \Vec$
is the functor given by $(X,Y)\mapsto F(X)\otimes F(Y)$.  
Thus to recover the coproduct of $H$, we need to define 
a homomorphism $\Delta: \End F\to  \End (F \otimes F)$. So 
for $a\in \End F$ and $X,Y \in \mathcal{C}$ we need to define a 
linear map 
\begin{equation}
\Delta(a)_{X,Y}: F(X)\otimes F(Y) \to F(X)\otimes F(Y).
\end{equation}
We define this map as the composition
\begin{equation}
F(X)\otimes F(Y) \xrightarrow{J_{X,Y}} F(X\otimes Y)  \xrightarrow{F(a_{X\otimes Y})} F(X\otimes Y) \xrightarrow{J_{X,Y}^{-1}} F(X)\otimes F(Y),
\end{equation}
where $J$ is the natural monoidal structure on $F$. 

\begin{proposition} Let $H$ be an associative unital algebra and  
$\mathcal{C}=H-{\rm mod}$. Suppose that $\mathcal{C}$ 
is equipped with a monoidal structure, and let  
$F:\, \mathcal{C} \to \Vec$ be an exact faithful monoidal functor ({\bf fiber functor}). Then $H=\End F$ is a bialgebra which is Hopf if $\mathcal{C}$ is rigid, and the functor $F$ gives rise to a monoidal equivalence $\mathcal{C}\cong \Rep(H)$. 
\end{proposition}

\subsubsection{} In particular, for symmetric categories we recover the classical Tannakian reconstruction theory for group schemes. 
\begin{example} (Classical Tannakian reconstruction)

(i) Let $\mathcal{C}=H-{\rm mod}$ be a rigid symmetric monoidal category over an algebraically closed field $\bold k$, and $F$ a symmetric monoidal functor (preserves $c_{X,Y}$). Then $H=\Aut F$ is cocommutative.

(ii) If $\mathcal{C}$ is finite, i.e., $H$ is finite dimensional, then $H = \bold k G$ is the group algebra of a finite group scheme $G$, and for any commutative $\bold k$-algebra $R$, 
\begin{equation}
G(R) = \{g\in (H\otimes_{\bold k}R)^\times~|~ \Delta(g) = g \otimes g \},
\end{equation}
the group of grouplike elements of $H$. 
In this case $G \simeq \Aut_{\otimes} (F)$ as a group scheme. Moreover, if $\dim H\ne 0$ in $\bold k$ (e.g., if ${\rm char}(\bold k)=0$) then $G$ is reduced, i.e., a usual finite group (Subsection \ref{prob1}, Problem 6).   
\end{example}

\subsubsection{}
Let $\mathcal{C}=H-{\rm mod}$ be a braided rigid monoidal category and $F:\mathcal{C} \to \Vec$ a exact faithful monoidal functor which is not required to preserve the braiding. 
 What structure does the braiding induce on the Hopf algebra 
 $H=\End F$? 
 
To answer this question, let $c_{X,Y}: X\otimes Y \to Y\otimes X$ be the braiding in $\mathcal{C}$ and define $\R \in H\otimes H$ by the following diagram:
\begin{equation}
\begin{CD}
F(X)\otimes F(Y) @ >{P\circ \R}>> @ . F(Y)\otimes F(X)\\
@VV{J_{X,Y}}V @ . @ AA{J_{Y,X}^{-1}}A\\
F(X\otimes Y) @ >{F(c_{X,Y})}>> @ . F(Y\otimes X)
\end{CD}
\end{equation} 

The condition that the braiding is invertible then translates into the condition that 
$\R$ is invertible, the condition that it is a morphism translates into 
the condition that $P\circ \R$ commutes with $\Delta (a)\in H$, i.e., 
$$
\R\Delta(a) = \Delta^{\op}(a) \R,
$$
and the hexagon axioms $\hexd,\hexdm$ translate into the {\bf hexagon relations}
$$ 
(\Delta \otimes \id) \R = \R^{13}\R^{23},~ (\id\otimes\Delta) \R = \R^{13}\R^{12}.
$$

\begin{definition} An element $\R\in H\otimes H$ 
satisfying these properties is called a {\bf quasi-triangular structure} on $H$.
A Hopf algebra equipped with a quasitriangular structure is called a 
{\bf quasitriangular Hopf algebra}, and the element $\R$ is called 
the {\bf universal $\R$-matrix} of $H$.
\end{definition} 

Thus having a quasitriangular structure on $H$ is equivalent to having a braiding on $\Rep(H)$.
So we obtain

\begin{proposition} 
The quantum double $\D(H)$ of a finite dimensional Hopf algebra $H$ 
is a quasitriangular Hopf algebra with $\R=\sum_i a_i\otimes a_i^*$, where $\lbrace a_i\rbrace$ is a basis of $H$. 
\end{proposition} 

\begin{exercise} Show that a quasitriangular structure $\R$ satisfies the {\bf quantum Yang-Baxter equation}
$$
\R^{12}\R^{13}\R^{23}=\R^{23}\R^{13}\R^{12}.
$$   
\end{exercise}

\begin{remark}  A quasitriangular structure $\R$ is called {\bf triangular} if $\R^{21} \R = 1 \otimes 1$, where $\R^{21}$ is obtained from $\R$ by swapping its components.  
Thus $\R$ is triangular if and only if $c_{X,Y}\circ c_{Y,X}=1 \otimes 1$, i.e., the braiding $c$ is symmetric. 
\end{remark} 

\subsection{$U_q(\mathfrak{sl}_2)$ at a root of unity} (\cite{EGNO}, 5.6, \cite{CP}, 9.3, \cite{K}, VI.5, \cite{J2}, Ch. 1). 

\subsubsection{}
Assume that $q$ is a root of unity of odd order $\ell >1$. In this specialization there are several different versions of quantum $\mathfrak{sl}_2$. For this reason, at roots of unity the Hopf algebra $U_q(\mathfrak{sl}_2)$ defined above is called the {\bf De Concini-Kac quantum group}, and often denoted $U_q^{\rm DK}(\mathfrak{sl}_2)$ (to distinguish it from the {\bf Lusztig quantum group} 
$U_q^{\rm L}(\mathfrak{sl}_2)$ defined below). 

Another important feature of roots of unity is that $U^{\rm DK}_q(\mathfrak{sl}_2)$
 acquires a ``big center". Indeed, recall from Subsection \ref{prob1}, Problem 10 that for generic $q$ (not a root of unity) the center is generated by the quantum Casimir element: 
\begin{equation}
C = f e + \dfrac{qK + q^{-1} K^{-1}-2}{(q-q^{-1})^2}, ~ Z = \mathbb{C}[C]. 
\end{equation}
On the other hand, as shown by De Concini and Kac (in the more general setting of any simple Lie algebra), at a root of unity the center $Z$ is much larger, so that the quantum group is a finitely generated $Z$-module. 

\begin{exercise} Check that if $q$ is a primitive $\ell$-th root of $1$ then $e^\ell,f^{\ell},K^{\ell}\in Z$.
\end{exercise}

This implies that if $q$ is a primitive $\ell$-th root of $1$ then $U_q^{\rm DK}(\mathfrak{sl}_2)$ has a finite dimensional Hopf quotient. 

Namely, let us call an ideal $I$ in a Hopf algebra $H$ a {\bf Hopf ideal} if $\Delta(I)\subseteq I\otimes H + H\otimes I$, $S(I)\subseteq I$. The quotient $H/I$ of $H$ by such an ideal is a Hopf algebra.

Now take $I = \langle e^\ell,f^{\ell},K^{\ell}-1\rangle$. One can compute that 
\begin{equation}
\Delta (e^{\ell}) = e^{\ell} \otimes K^{\ell} + 1\otimes e^\ell,
~~~
\Delta (f^{\ell}) = f^{\ell} \otimes 1+ K^{-\ell} \otimes f^{\ell},
~~~
\Delta (K^{\ell}-1) = (K^{\ell}-1)\otimes K^{\ell} + 1 \otimes (K^{\ell}-1).
\end{equation}
Thus $I$ is a Hopf ideal and $u_q(\mathfrak{sl}_2):=H/I$ is a Hopf algebra. It is called the  
{\bf small quantum group} and was introduced by Lusztig (in the more general setting of any simple Lie algebra). It is easy to see that $\dim u_q(\mathfrak{sl}_2) = \ell^3$.

\subsubsection{}
Following Lusztig, there is another definition of the small quantum group. Namely, define the 
{\bf Lusztig quantum group} $U_v^{\mathrm{L}}(\mathfrak{sl}_2)$ over $\mathbb{C}[v,v^{-1}]$ as the $\mathbb{C}[v,v^{-1}]$-subalgebra in the usual quantum group $U_v(\mathfrak{sl}_2)_{\Bbb C(v)}$
over the field $\Bbb C(v)$ generated by the elements $K^{\pm 1}$ and {\bf divided powers} 
\begin{equation}
e^{(n)} := \dfrac{e^n}{[n]_v !},~~~
f^{(n)} := \dfrac{f^n}{[n]_v !},~~~
n\ge 1. 
\end{equation}

Let 
$$
h^{(\ell)}:=\frac{\prod_{j=0}^{\ell-1}(Kv^{-j}-K^{-1}v^j)}{[\ell]_v!}.
$$

\begin{exercise}\label{lqg} Show that the algebra $U_v^{\mathrm{L}}(\mathfrak{sl}_2)$ specialized at a primitive $\ell$-th root of unity $v=q$ is actually generated by $e,f,K^{\pm 1}$ and the specializations of the elements $e^{(\ell)},f^{(\ell)},h^{(\ell)}$.
\end{exercise}

In fact, one can write defining relations for these generators (but we won't give them here). 

Note that since $[\ell]_q=0$, we have 
$$
[\ell]_qh^{(\ell)}=\prod_{j=0}^{\ell-1}(Kq^{-j}-K^{-1}q^j)=0,
$$ 
hence 
$K^{2\ell}=1$ in $U_q^{\mathrm{L}}(\mathfrak{sl}_2)$. 
Moreover, it is easy to see that $K^\ell\in U_q^{\mathrm{L}}(\mathfrak{sl}_2)$ is a central element (of order $2$). So we may consider the quotient $\overline U_q^{\mathrm{L}}(\mathfrak{sl}_2)$ of 
$U_q^{\mathrm{L}}(\mathfrak{sl}_2)$ by the relation $K^\ell=1$. 

\begin{proposition} $U_q^{\mathrm{L}}(\mathfrak{sl}_2)$ and $\overline U_q^{\mathrm{L}}(\mathfrak{sl}_2)$ are Hopf algebras, and 
$\mathfrak{u}_q(\mathfrak{sl}_2)$ is a Hopf subalgebra of $\overline U_q^{\mathrm{L}}(\mathfrak{sl}_2)$ generated by $e,f,K^{\pm 1}$. Thus we have a diagram of Hopf algebra maps 
\begin{equation}
U_q^{\rm DK}(\mathfrak{sl}_2) ~\twoheadrightarrow~ u_q(\mathfrak{sl}_2) ~ \hookrightarrow ~ \overline U_q^{\mathrm{L}}(\mathfrak{sl}_2). 
\end{equation}
\end{proposition} 

\subsubsection{}
Note that the universal $R$-matrix of $U_v(\mathfrak{sl}_2)$ cannot be specialized to 
a root of unity $q$ because of the poles in the formula for $\R$
\begin{equation}
\R = v^{\frac{h\otimes h}{2}}~\sum\limits_{k=0}^{\infty} v^{\frac{k(k-1)}{2}}\dfrac{(v-v^{-1})^k}{[k]_v !}e^{k}\otimes f^{k}
\end{equation}
in the limit $v\to q$. However, using the divided powers $e^{(k)},f^{(k)}$, we may rewrite this formula as 
\begin{equation}
\R = v^{\frac{h\otimes h}{2}}~\sum\limits_{k=0}^{\infty} v^{\frac{k(k-1)}{2}}[k]_v !(v-v^{-1})^ke^{(k)}\otimes f^{(k)}.
\end{equation}
So instead of poles we now have zeros. It follows that when $v=q$, this series makes sense and moreover terminates at the $\ell$-th term. As a result, $u_q(\mathfrak{sl}_2)$ and 
$\overline U_q^{\mathrm{L}}(\mathfrak{sl}_2)$ are quasi-triangular Hopf algebras with 
the $\R$-matrix given by the truncation of this series at the $\ell$-th term: 
\begin{equation}\label{drfor} 
\R = \Theta \cdot\sum_{k=0}^{\ell -1}q^{\frac{k(k-1)}{2}}\dfrac{(q-q^{-1})^k}{[k]_q !}e^{k}\otimes f^{k},
\end{equation} 
where $\Theta\in (\Bbb C[K]/(K^\ell-1))^{\otimes 2}$ 
is the specialization of $v^{\frac{h\otimes h}{2}}$ which is defined by the formula 
$$
\Theta:=\frac{1}{\ell}\sum_{i,j=0}^{\ell-1}q^{-2ij}K^i\otimes K^j.
$$

\subsubsection{} The behavior of quantum groups at a root of $1$ is analogous to the behavior of algebraic groups and Lie algebras in characteristic $p$. Namely, let $\bold k$ 
be an algebraically closed field of characteristic $p>2$, $U(\mathfrak{sl}_2)$ the usual enveloping algebra of $\mathfrak{sl}_2$ over $\bold k$, and $\mathrm{SL}_2(\bold k)$ the corresponding group. Like in characteristic zero, we have a functor $ \Rep (\mathrm{SL}_2(\bold k))\to \Rep(\mathfrak{sl}_2)=\Rep(U(\mathfrak{sl}_2))$ (derivative at $1$), but unlike characteristic zero, no functor in the opposite direction. The reason is that the power series for the exponential function has zeros in the denominators, so in general we cannot exponentiate a linear transformation, even if it is nilpotent, and thus cannot integrate representations of the Lie algebra to the group.  

There is, however, a fix: we can replace the usual enveloping algebra $U(\mathfrak{sl}_2)$ 
by its divided power version, called the {\bf distribution algebra}. 
Namely, consider the subalgebra $U_{\mathbb{Z}}^{\rm K}(\mathfrak{sl}_2)$ inside $U_{\mathbb{Q}}(\mathfrak{sl}_2)$ generated by 
\begin{equation}
e^{(n)} =\dfrac{e^n}{n!}, ~~~ f^{(n)} =\dfrac{f^n}{n!}, ~~~ h^{(n)} = \binom{h}{n}=\frac{h(h-1)...(h-n+1)}{n!}, n\ge 1 \text{ (Kostant's $\mathbb{Z}$-form)}. 
\end{equation}
Then the {\bf distribution algebra} is the Hopf algebra $U_{\bold k}^{\rm K}(\mathfrak{sl}_2) = U_{\mathbb{Z}}^{\rm K}(\mathfrak{sl}_2)\otimes_{\mathbb{Z}} {\bold k}$, which can be thought of as the algebra of distributions on $\mathrm{SL}_2(\bold k)$ concentrated at $1$, i.e., 
linear functionals $\bold k[\mathrm{SL}_2]\to \bold k$ which annihilate some power of the maximal ideal of $1$, 
with the product and coproduct dual to those of $\bold k[\mathrm{SL}_2]$. One can show that  
finite dimensional representations of the distribution algebra are the same as representations of the algebraic group $\mathrm{SL}_2(\bold k)$. 

Also the usual enveloping algebra $U(\mathfrak{sl}_2)$ develops a ``big center": in addition to the usual Casimir $C=fe+\frac{(h+1)^2}{4}$ we have the central elements $e^p,f^p,h^p-h$. These elements are primitive ($\Delta(x)=x\otimes 1+1\otimes x)$, so they generate a Hopf ideal and we get a finite dimensional cocommutative Hopf algebra $U(\mathfrak{sl}_2)/(e^{p},~f^{p},~ h^p-h)=u(\mathfrak{sl}_2)$. This Hopf algebra is called the {\bf restricted enveloping algebra}. It sits as a Hopf subalgebra generated by $e,f,h$ inside $U_{\bold k}^{K}(\mathfrak{sl}_2)$, and has dimension $p^3$. Thus we have a diagram of Hopf algebra maps
\begin{equation}
U(\mathfrak{sl}_2) ~\twoheadrightarrow~ u(\mathfrak{sl}_2) ~ \hookrightarrow ~ U_{\bold k}^{\rm K}(\mathfrak{sl}_2). 
\end{equation}
We see that the $\Bbb Z$-form $U_{\bold k}^{\rm K}(\mathfrak{sl}_2)$ is analogous to the Lusztig form $U_q^{\mathrm{L}}(\mathfrak{sl}_2)$ of the quantum group, the enveloping algebra $U_{\bold k}(\mathfrak{sl}_2)$ to the De Concini-Kac quantum group $U_q^{\rm DK}(\mathfrak{sl}_2)$, and the restricted enveloping algebra $u(\mathfrak{sl}_2)$ to the small quantum group $u_q(\mathfrak{sl}_2)$. Moreover when $q$ is a primitive $p$-th root of unity, one can obtain the classical algebras from the quantum ones by reduction to characteristic $p$. 

\subsection{Quantum Frobenius} (\cite{CP}, 9.3, \cite{J2}, Ch. 1, \cite{Lu}).
\subsubsection{} The quantum Frobenius map 
is a quantum analog of the classical Frobenius map for algebraic groups in characteristic $p$, so let us recall the classical Frobenius map first. 

Let $G=\mathrm{SL}_2(\bold k)$, where $\mathrm{char}(\bold k) = p$. The classical (geometric) Frobenius automorphism ${\rm F}$ acts on $\bold k[G]$ by
\begin{equation}
\left(
\begin{array}{cc}
a & b \\
c & d
\end{array}
\right)
\mapsto
\left(
\begin{array}{cc}
a^p & b^p \\
c^p & d^p
\end{array}
\right).
\end{equation} 
There is a dual map ${\rm F}^*:~U_{\bold k}^{\rm K}(\mathfrak{sl}_2)\to U_{\bold k}^{\rm K}(\mathfrak{sl}_2)$ defined by
\begin{equation}
\def\arraystretch{1.5}
\begin{array}{llll}
e \mapsto 0, & e^{(p)} \mapsto e, & e^{(p^2)} \mapsto e^{(p)}, & ...\\
f \mapsto 0, & f^{(p)} \mapsto f, & f^{(p^2)} \mapsto f^{(p)}, & ...\\
h \mapsto 0, & h^{(p)} \mapsto h, & h^{(p^2)} \mapsto h^{(p)}, & ...\\
\end{array}
\end{equation}
Thus the sequence of Hopf algebra maps $U(\mathfrak{sl}_2) ~\twoheadrightarrow~ u(\mathfrak{sl}_2) ~ \hookrightarrow ~ U_{\bold k}^{\rm K}(\mathfrak{sl}_2)$ discussed above can be extended to the sequence 
\begin{equation}\label{cfrob}
U(\mathfrak{sl}_2) ~\twoheadrightarrow~ u(\mathfrak{sl}_2) ~ \hookrightarrow ~ U_{\bold k}^{\rm K}(\mathfrak{sl}_2) ~\twoheadrightarrow~ U_{\bold k}^{\rm K}(\mathfrak{sl}_2),
\end{equation}
where the last map is ${\rm F}$. The dual of the last three terms of this sequence geometrically is 
the short exact sequence of affine group schemes 
\begin{equation}
1 ~\to~ G_{(1)} ~\to~ G ~\xrightarrow{\rm F}~ G ~\to~ 1
\end{equation}
Here $G_{(1)}$ is the infinitesimal group scheme called the {\bf Frobenius kernel}, such that 
${\bold k}[G_{(1)}]=u(\mathfrak{sl}_2)$ and $\mathcal O(G_{(1)})=u(\mathfrak{sl}_2)^*$. 

\subsubsection{}
The quantum Frobenius automorphism is defined by quantizing the definition of the usual Frobenius map. Namely, note that the ideal generated by $e,f,K-1$ inside $\overline U_q^{L}(\mathfrak{sl}_2)$ (i.e., by the kernel of the counit of the small quantum group) is a Hopf ideal. 
So we may consider the Hopf algebra 
\begin{equation}
A:=U_q^{L}(\mathfrak{sl}_2)/(e,f,K-1) = \langle E:=e^{(\ell)},F:=f^{(\ell)},H:=h^{(\ell)} \rangle.
\end{equation}
\begin{proposition}\label{qfrob} The Hopf algebra $A$ is isomorphic to $U(\mathfrak{sl}_2)$. Namely, the elements $E,F,H$ satisfy the usual defining relations
\begin{equation}
[H,E] = 2E,\ [H,F] = -2 F,\ [E,F] = H.  
\end{equation}
\end{proposition} 

Thus the sequence of Hopf algebra maps $U_q^{\mathrm{DK}}(\mathfrak{sl}_2) ~\twoheadrightarrow~ u_q(\mathfrak{sl}_2) ~ \hookrightarrow ~ U_q^{\mathrm{L}}(\mathfrak{sl}_2)$
considered above can be extended to the sequence 
\begin{equation} \label{qfrob1}
U_q^{\mathrm{DK}}(\mathfrak{sl}_2) ~\twoheadrightarrow~ u_q(\mathfrak{sl}_2) ~ \hookrightarrow ~ U_q^{\mathrm{L}}(\mathfrak{sl}_2) \xrightarrow{\rm F} U(\mathfrak{sl}_2),
\end{equation}
where the last map ${\rm F}$ comes from Proposition \ref{qfrob} and is called 
{\bf the quantum Frobenius map}. 

The sequence \eqref{qfrob1} is analogous to the classical  sequence 
\eqref{cfrob}, and in fact turns into it when the order of $q$ is $p$ and we reduce to characteristic $p$. 
For this reason $u_q(\mathfrak{sl}_2)$ is often called the {\bf Frobenius-Lusztig kernel}.  
Note, however, a crucial difference -- unlike the classical Frobenius, the quantum Frobenius cannot be iterated, i.e., ${\rm F}^m$ does not make sense for $m>1$. Indeed, ${\rm F}$ already lands in the usual enveloping algebra (over $\Bbb C$), which does not have a Frobenius map. This is so for any simple Lie algebra, and is the reason why the representation theory of quantum groups at roots of unity, while in some essential ways similar to that of algebraic groups in characteristic $p$, is substantially simpler. 

\begin{remark} One can informally view the last three terms of the sequence  
\eqref{qfrob1} (or, rather, the dual sequence of quantized function algebras) as a ``short exact sequence of quantum groups":
$$
1\to G_{(1)}^q\to G^q\xrightarrow{\rm F} G\to 1
$$
(where only the last term is a usual group, $G=\mathrm{SL}_2(\Bbb C)$). 
One can show that analogously to Exercise \ref{shortex}(iii), this gives rise to an action 
of $G$ on $\Rep(G_{(1)}^q)=\Rep(u_q(\mathfrak{sl}_2))$, and one has 
$$
\Rep(u_q(\mathfrak{sl}_2))^G\cong \Rep(G^q)\cong \Rep(\overline{U}^{\rm L}_q(\mathfrak{sl}_2)).
$$ 
\end{remark} 

\subsection{The universal $R$-matrix for the quantum $\mathfrak{sl}_2$} (\cite{K}, Ch. VIII, XI, \cite{CP}, Ch. 8).\label{qdsl}

\subsubsection{}
We are now ready to give a motivated construction (from scratch) of $u_q(\mathfrak{sl}_2)$ and its universal $R$-matrix. 
Let $\mathfrak{b}_+$, $\mathfrak{b}_-$ be the positive and negative 
Borel subalgebras of $\mathfrak{sl}_2$ spanned by $h,e$ and $h,f$ respectively. 
The quantum analog of $\mathfrak{b}_+$ is generated by one grouplike element and one skew-primitive element on which the grouplike element acts with eigenvalue $q^2$: 
\begin{equation}
U_q(\mathfrak{b}_{+}) =
\langle
K_{+},e \, | \, K_{+}eK_{+}^{-1}=q^2 e 
\rangle
,~~~
\Delta(K_{+}) = K_{+} \otimes K_{+}, ~ \Delta(e) = e\otimes K_{+} + \M1 \otimes e. 
\end{equation}
Similarly, we can define the (in fact, isomorphic) Hopf algebra 
\begin{equation}
U_q(\mathfrak{b}_{-}) =
\langle
K_{-},f \, | \, K_{-}fK_{-}^{-1}=q^{-2} f
\rangle
,~~~
\Delta(K_{-}) = K_{-} \otimes K_{-}, ~ \Delta(f) = f\otimes \M1 + K_{-}^{-1} \otimes f. 
\end{equation}
Now, if $\mathrm{ord}(q) = \ell ~ $ then  
$$
I_+=\langle K_+^{\ell} - 1,~e^\ell\rangle \subset U_q(\mathfrak{b}_{+})
$$ 
is a Hopf ideal, and we can define the Taft algebra $u_q(\mathfrak{b}_+):=U_q(\mathfrak{b}_{+})/I_+$ which  is a Hopf algebra with $\dim u_q(\mathfrak{b}_+) = \ell^2$ and linear basis $K_+^i e^j,~ 0\leq i,j \leq \ell-1$ (see Subsection \ref{prob1}, Problem 3). Similarly, we can define the Hopf ideal 
$$
I_-=\langle K_-^{\ell} - 1,\ f^\ell\rangle\subset U_q(\mathfrak{b}_{-})
$$ 
and define the Taft algebra $u_q(\mathfrak{b}_-):=U_q(\mathfrak{b}_{-})/I_-$, which  is a Hopf algebra with $\dim u_q(\mathfrak{b}_-) = \ell^2$ and linear basis $K_-^i f^j,~ 0\leq i,j \leq \ell-1$.  

\subsubsection{}
\begin{lemma} One has a Hopf algebra isomorphism $u_q(\mathfrak{b}_{+})^{*\cop} \cong u_q(\mathfrak{b}_{-})$.
\end{lemma} 

\begin{proof} (sketch) To construct such an isomorphism, we need
a nondegenerate {\bf Hopf pairing}
$$
 (,): u_q(\mathfrak{b}_{+})\otimes u_q(\mathfrak{b}_{-})\to \mathbb{C},
$$
i.e., pairing satisfying the equalities 
\begin{equation}
(\Delta (x),z\otimes y) = (x,yz),~~~ (y\otimes z,\Delta (x)) = (yz,x). 
\end{equation}
And one can show that there exists a unique such pairing determined by the equalities 
\begin{equation}
(K_+,K_-) = q^{2},~~~ (e,f) = \dfrac{1}{q-q^{-1}},~~~ (K_+,f)=(e,K_-)=0.
\end{equation}
\end{proof} 

Thus the Drinfeld double of $u_q(\mathfrak{b}_{+})$ as a vector space has the form
\begin{equation}
\D (u_q(\mathfrak{b}_{+})) = u_q(\mathfrak{b}_{+})\otimes u_q(\mathfrak{b}_{+})^{*\cop} = u_q(\mathfrak{b}_{+}) \otimes u_q(\mathfrak{b}_{-}). 
\end{equation}

\begin{exercise} The element $K_c=K_+ K_-^{-1}$ is central in 
$\D (u_q(\mathfrak{b}_{+}))$, and $\Delta(K_c)=K_c\otimes K_c$. 
\end{exercise} 

Thus the ideal generated by $K_c-1$ is a Hopf ideal in $\D (u_q(\mathfrak{b}_{+}))$. So we may consider the quotient Hopf algebra $\D (u_q(\mathfrak{b}_{+}))/(K_c-1)$, reducing dimension from $\ell^4$ to $\ell^3$. The images of $K_+$ and $K_-$ in this quotient are the same, so let us denote
this element by $K$. 

\begin{exercise} (i) Show that in $\D (u_q(\mathfrak{b}_{+}))/(K_c-1)$ one has 
\begin{equation}
[e,f] = \dfrac{K-K^{-1}}{q-q^{-1}}.
\end{equation}
Deduce that 
$$
\D (u_q(\mathfrak{b}_{+}))/(K_c-1) \cong u_q(\mathfrak{sl}_2).
$$

(ii) Moreover, show that 
$\D (u_q(\mathfrak{b}_{+})) \cong u_q(\mathfrak{sl}_2)\otimes \mathbb{C}[\mathbb{Z}/\ell \mathbb{Z}]$ as Hopf algebras, where $\mathbb{C}[\mathbb{Z}/\ell \mathbb{Z}]$ is generated by $K_c$.
\end{exercise} 

\subsubsection{} Now we are ready to compute the universal $R$-matrix of $u_q(\mathfrak{sl}_2)$. 
Recall that in $\D (u_q(\mathfrak{b}_{+}))$ we have 
\begin{equation}
\label{RUniv}
\R = \sum_i a_i \otimes a_i^*,
\end{equation}
where $\lbrace a_i\rbrace$ and $\lbrace a_i^*\rbrace$ are dual bases of $u_q(\mathfrak{b}_{+})$ and $u_q(\mathfrak{b}_{-})$ under the Hopf pairing $(,)$. So to obtain an explicit formula for $\R$, we need 
to compute this Hopf pairing.  
 
\begin{exercise} Show that 
\begin{equation}
(K_+^i,K_-^j) = q^{2ij},~~~ (e^n,f^n) = \dfrac{[n]_q!}{(q-q^{-1})^n}q^{-\frac{n(n-1)}{2}}.
\end{equation}
Deduce from this Drinfeld's formula \eqref{drfor} for the $R$-matrix of $u_q(\mathfrak{sl}_2)$.\end{exercise} 

\begin{remark} 
A similar construction of the quantum group and its $R$-matrix 
works for $U_q(\mathfrak{sl}_{2})$ for generic (or formal) $q$.
However, there are some subtleties related to the fact that $U_q(\mathfrak{b}_+)$ is infinite dimensional. Namely, $H=U_q(\mathfrak{sl}_{2})$ is not literally quasitriangular, as 
the $R$-matrix belongs not to $H\otimes H$ but only to its completion $H\widehat\otimes H$ (since the sum defining $\R$ is infinite). So one has to consider restricted dual algebras and the restricted Drinfeld double. 

Also, as a result, the category of all representations of $U_q(\mathfrak{sl}_2)$ is not a braided monoidal category, as the series for $\R$ does not converge when evaluated in $V\otimes W$ for general representations $V,W$. However, if representations $V,W$ of $U_q(\mathfrak{sl}_{2})$ are locally nilpotent under $e$ (i.e. $\forall ~v~ \exists N:~e^{N}v=0$) then $\R|_{V\otimes W}$ makes sense (as the series for $\R$ terminates at some place), so such representations do form a braided monoidal category. In particular, this category includes all finite dimensional representations. 
\end{remark} 

\subsection{Problems}\label{prob2}
\begin{enumerate}
\item Show that tensor autoequivalences of $\Vec(G)$ which act 
trivially on isomorphism classes of simple objects  are classified up to isomorphism by $H^2(G,\mathbb{C}^\times)$.
\item Show that tensor automorphisms of the identity functor on $\Vec (G)$ are classified by $$
H^1(G,\mathbb{C}^\times) = \Hom(G,\mathbb{C}^\times).
$$
\item Show that $\Bbb C$-linear monoidal functors $\Vec(G) \xrightarrow{F} \Vec$ are classified by $H^2(G,\mathbb{C}^\times)$.
\item Let $H$ be a Hopf algebra, $F:\Rep(H) \to \Vec$ -- the forgetful functor and $J\in H \otimes H$. What are the conditions on $J$ ensuring that it defines a tensor structure on $F$?
\item Let $H=\mathbb{C}G$ and $J$ is as in Problem 4 and is symmetric ($J\in S^2H$). Show that $\exists$ invertible $a\in\mathbb{C}G$ such that $J = \Delta(a)(a^{-1}\otimes a^{-1})$.
\item Compute the coproduct of $e^{(\ell)},~f^{(\ell)}$ in $U_q^{\rm L}(\mathfrak{sl}_2)$.
\item Consider the quotient $A = U_q^L(\mathfrak{sl}_2)/(e,f,K-1)$. Show that $A\cong U(\mathfrak{sl}_2)$ (i.e., prove Proposition \ref{qfrob}).
\item Describe finite dimensional, irreducible representations of $U_q^{\rm DK}(\mathfrak{sl}_2)$
when $q$ is root of unity of odd order $\ell\ge 3$ (Hint: fix the eigenvalues of the central elements 
$e^\ell,f^\ell,K^\ell-1$ first). 
\item Show that when $q$ is root of unity of odd order $\ell\ge 3$ then irreducible 
finite dimensional representations of $\overline U_q^{\rm L}({\mathfrak{sl}}_2)$ are parametrized by 
highest weights $m\in \Bbb Z_+$. Namely, for any such $m$ there is a unique 
irreducible finite dimensional representation $L_m$ with a vector $v\ne 0$
such that $Kv=q^mv$ and $ev=e^{(\ell)}v=0$, and any irreducible finite dimensional  
representation is isomorphic to $L_m$ for exactly one $m$. 
\item What are irreducible representations of $u_q(\mathfrak{sl}_2)=U_q(\mathfrak{sl}_2)/(e^\ell, f^\ell, K^\ell -1)$?
Compute the composition series of the tensor product of two such representations.
\item Let $R=\mathbb{Z}[q],~ q=e^{2\pi i/p}$ where $p>2$ is a prime, and $\chi: R \to \mathbb{F}_p$ be the homomorphism defined by $\chi(q) = 1$. Define all versions of the quantum $\mathfrak{sl}_2$ over $R$ and show that they go to the corresponding versions of the enveloping algebra in characteristic $p$ under the functor $\cdot \bigotimes\limits_R \mathbb{F}_p$, where 
$R$ acts on $\Bbb F_p$ via $\chi$. 
\end{enumerate}

\section{Quantum universal enveloping algebras}

\subsection{The quantum group $U_q(\mathfrak{g})$} (\cite{J}, Ch. 4,5, \cite{Lu}) 
\subsubsection{} We now describe a generalization of the quantization of $\mathfrak{sl}_2$ considered above
to more general Lie algebras. Let $\mathfrak{g}$ be a simple complex Lie algebra, $\mathfrak{b}_+$ be a Borel subalgebra in $\mathfrak{g}$, and 
$(a_{ij}, 1\le i,j\le r)$ its Cartan matrix. Recall that this matrix is symmetrizable: 
there are unique numbers $d_i=1,2,3$, $i\in [1,r]$ (with the smallest one being $1$) such that 
\begin{equation}
d_i a_{ij} = d_ja_{ji}.
\end{equation}
Let  $q_i := q^{d_i}$.
Consider the Hopf algebra
\begin{equation}
\widetilde{U_q(\mathfrak{b}_+)} : = \langle \{K_i\}, \{e_i\} ~ | ~ K_i K_j = K_j K_i,~ K_i e_j K_i^{-1} = q_i^{a_{ij}}e_j \rangle
\end{equation}
where the coproduct, counit and antipode is defined on each subalgebra generated by $K_i,e_i$ as in the $\mathfrak{sl}_2$ case. 
Let $I$ be the largest Hopf ideal of $\widetilde{U_q(\mathfrak{b}_+)}$ not containing any $e_i$, and define  the {\bf quantum Borel subalgebra} $U_q(\mathfrak{b}_+) := \widetilde{U_q(\mathfrak{b}_+)}/I$.

\begin{theorem} (G. Lusztig) The ideal $I$ is generated by the {\bf quantum Serre relations}
\begin{equation}
\sum_{k=0}^{1-a_{ij}} \binom{1-a_{ij}}{k}_{q_i} e_i^k e_j e_i^{1-a_{ij}-k} = 0,\ i\ne j. 
\end{equation}
where $\binom{n}{k}_q:=\frac{[n]_q!}{[k]_q![n-k]_q!}$.
\end{theorem} 

Now similarly to Subsection \ref{qdsl} we can define the {\bf restricted quantum double}
$\D (U_q(\mathfrak{b}_{+}))_{\mathrm{res}}$, which is the Hopf algebra
$U_q(\mathfrak{b}_{+}) \otimes U_q(\mathfrak{b}_{-})$ with the twisted product rule, and
\begin{equation}
\D (U_q(\mathfrak{b}_{+}))_{\mathrm{res}}/(K_{ic}-1, i=1,...,r) = U_q(\mathfrak{g})
\end{equation}
with $K_{ic} = K_{i+}K_{i-}^{-1}$. In this quotient $K_{i+}=K_{i-}=K_i$ 
and besides the relations coming from the two halves, we have the commutation relation coming from the twisted product condition: 
$$
[e_i,f_j]=\delta_{ij}\frac{K_i-K_i^{-1}}{q_i-q_i^{-1}}.
$$
All these relations taken together are defining for $U_q(\mathfrak{g})$. 
Also, as before, the quantum double construction automatically produces 
the universal $R$-matrix. 

\begin{remark} 1. These constructions make sense more generally, for $\mathfrak{g}$ a symmetrizable Kac-Moody algebra (the only difference is that the symmetrized Cartan matrix 
$d_ia_{ij}$ is not necessarily positive definite, with slight changes in the Cartan part if 
it is degenerate). 

2. There exists a more explicit formula, due to Levendorskii and Soibelman, for the universal $R$-matrix of $U_q(\mathfrak{g})$ for a simple Lie algebra $\mathfrak g$, given, roughly speaking, by a noncommutative product over positive roots of $R$-matrices of $\mathfrak{sl}_2$ for these roots. This formula can be found in the textbook \cite{KS}. There are also generalizations of this formula to the case of affine Lie algebras, see \cite{LSS,KT}.
\end{remark} 

 Furthermore, as in the $\mathfrak{sl}_2$ case, at a root of unity $q$ (of odd order, coprime to $3$ for $G_2$) one can define the small quantum group $u_q(\mathfrak{g})$, which is the quantum analog of the restricted enveloping algebra $u(\mathfrak{g})$ in characteristic $p$. However, in the higher rank case in $u(\mathfrak{g})$ we have to impose the relations $e_\alpha^p=0$ for all (not just simple) roots $\alpha$. Thus in the quantum situation we first have to define the quantum analogs of non-simple root elements $e_\alpha$. This can be done for general $q$ using {\bf Lusztig's braid group action} on $U_q(\mathfrak{g})$, see \cite{Lu,J}, but is quite nontrivial and will not be discussed here. 

At roots of unity one can also define {\bf Lusztig's divided power algebra} $\overline U_q^L(\mathfrak{g})$ 
and the quantum Frobenius map ${\rm F}: \overline U_q^L(\mathfrak{g})\to U(\mathfrak{g})$. Thus we have 
the diagram of Hopf algebra maps 
\begin{equation}
U_q^{\mathrm{DK}}(\mathfrak{g}) ~\twoheadrightarrow~ u_q(\mathfrak{g}) ~ \hookrightarrow ~ U_q^{\mathrm{L}}(\mathfrak{g}) \xrightarrow{\rm F} U(\mathfrak{g}).
\end{equation}

\begin{remark} All these objects can be defined for roots of unity of any order, but in the case of even order the details are somewhat different. For example, for roots of unity of order divisible by $4$ the quantum Frobenius map lands not in $U(\mathfrak{g})$ but in $U(\mathfrak{g}^\vee)$, where 
$\mathfrak{g}^\vee$ is the Langlands dual Lie algebra to $\mathfrak{g}$. 
\end{remark} 

\begin{remark} One may ask what kind of quantum group is produced from $U_q^{\rm DK}(\mathfrak{b}_+)$ 
by the quantum double construction when $q$ is a root of unity. It turns out that this is the {\bf mixed quantum group}, intermediate between the Lusztig and De Concini-Kac form. In this mixed form, 
negative root generators come with divided powers, while positive ones with usual powers only. 
The mixed form has a universal $R$-matrix which is given by an infinite sum, same as 
for generic $q$. It is discussed in \cite{G}, 5.3. 
\end{remark} 

\subsection{Classical limit} (\cite{CP}, Ch. 1-3, \cite{ES}, Ch. 2-4, 9). 

\subsubsection{}
Let us explain in what sense $U_q(\mathfrak{g})$ is a deformation of the usual enveloping algebra $U(\mathfrak{g})$. To this end, set $q:=e^{\frac{\hbar}{2}}$ and let us work over the base ring $\mathbb{C}[[\hbar]]$. Then the ($\hbar$-adically completed) quantum universal enveloping algebra $U_q(\mathfrak{g})$ turns out to be a {\bf formal deformation} of $U(\mathfrak{g})$ as Hopf algebra.
Such a deformation is called a {\bf quantized universal enveloping algebra.}

Namely, following Drinfeld, we make the following definition. 
\begin{definition} A quantized universal enveloping (QUE) algebra is 
a Hopf algebra $A$ over $\mathbb{C}[[\hbar]]$ (torsion-free, separated and complete as a module) together with a Hopf algebra isomorphism $A/\hbar A \simeq U(\mathfrak{g})$, where $\mathfrak{g}$ is a Lie algebra.\footnote{Here we use the completed tensor product of  $\mathbb{C}[[\hbar]]$-modules.} 
\end{definition} 

Let $A$ be a QUE algebra deforming $U(\mathfrak{g})$. Define
the map $\delta:~U(\mathfrak{g}) \to U(\mathfrak{g})\otimes U(\mathfrak{g})$ by the formula 
\begin{equation}
\delta(x):=\dfrac{\Delta(\widetilde x)-\Delta^{\op}(\widetilde x)}{\hbar}\text{ mod }\hbar,
\end{equation}
where $\widetilde x$ is a lift of $x$ to $A$. It is easy to show that $\delta(x)$ is well defined 
(independent on the lift $\widetilde{x}$). The map $\delta$ measures the non-cocommutativity of the coproduct in first order, and is called the {\bf co-Poisson-Hopf} structure on $U(\mathfrak{g})$ induced by $A$. Moreover, one can show that the restriction $\delta|_{\mathfrak{g}}$ maps 
$\mathfrak{g}$ to $\wedge^2\mathfrak{g}\subset U(\mathfrak{g})\otimes U(\mathfrak{g})$. This map 
\begin{equation}
\delta|_{\mathfrak{g}}:~ \mathfrak{g} \to \Lambda^2 \mathfrak{g}
\end{equation}
is called the {\bf cobracket} or {\bf Lie bialgebra structure} on $\mathfrak{g}$ induced by $A$.

\begin{proposition}\label{cobra}
(i) $\delta^{*}:~ \Lambda^2 \mathfrak{g}^* \to \mathfrak{g}^*$ is a Lie bracket;

(ii) $\delta$ is a derivation:
$$
\delta([x,y]) = [x\otimes \M1 + \M1 \otimes x, \delta(y)] + [\delta(x),y\otimes \M1 + \M1 \otimes y],
$$
i.e., $\delta \in Z^{1}(\mathfrak{g},\mathfrak{g}\otimes \mathfrak{g})$ is a 1-cocycle in the Chevalley–Eilenberg complex with coefficients in $\mathfrak{g} \otimes \mathfrak{g}$. 
\end{proposition}

\subsubsection{}
\begin{definition} A pair $(\mathfrak{g},\delta)$ with the properties of Proposition \ref{cobra} is called a {\bf Lie bialgebra}. The Lie bialgebra $(\mathfrak{g},\delta)$ arising from the QUE algebra $A$ is called {\bf the quasiclassical limit} of $A$. The QUE algebra $A$ is called a {\bf quantization} of $(\mathfrak{g},\delta)$.\\
\end{definition} 

\begin{exercise} Let $q=e^{\frac{\hbar}{2}}$. Show that the $\hbar$-adically complete subalgebra $A\subset U_q(\mathfrak{sl}_2)[\hbar^{-1}]$ generated by $e,f$ and $h:=\frac{K-K^{-1}}{q-q^{-1}}$ is a QUE algebra which deforms $U(\mathfrak{sl}_2)$ with cobracket given by 
$$
\delta(e)=e\wedge h, ~ \delta(f)=f\wedge h, ~ \delta(h)=0.
$$
\end{exercise} 

\begin{theorem}\label{EK1} (\cite{EK1}, see also \cite{ES}, Ch. 21) Over a field of characteristic zero, any Lie bialgebra has a quantization, and moreover this quantization can be given by a functor from the category of Lie bialgebras to the category of quantized universal enveloping algebras.
\end{theorem} 

\subsection{Classical quasitriangular structures} (\cite{ES}, Ch. 3; \cite{CP}, Ch. 2).

Let $A$ be a quantum universal enveloping algebra such that $A/\hbar A = U(\mathfrak{g})$, for a Lie bialgebra $\mathfrak{g}$, and suppose $A$ has a quasitriangular structure (universal $R$-matrix) $\mathcal R\in A\otimes A$ such that $\R=1\otimes 1$ modulo $\hbar$. Consider the element 
\begin{equation}
r = \dfrac{\R-1\otimes 1}{\hbar}\text{ mod }\hbar\in U(\mathfrak{g})\otimes U(\mathfrak{g}).
\end{equation}

\begin{proposition} (i) $r\in \mathfrak{g} \otimes \mathfrak{g}$;

(ii)  $\delta$ is the coboundary of $r$, i.e., 
\begin{equation}
\delta(x) = [x\otimes \M1 + \M1 \otimes x ,r].
\end{equation}
\end{proposition} 

In the order $\hbar^2$ the quantum Yang-Baxter equation for $\R$ 
gives the {\bf classical Yang-Baxter equation}
\begin{equation}
[r^{12},r^{13}]+[r^{12},r^{23}]+[r^{13},r^{23}]=0. 
\end{equation}

Also the element $r+r^{21}=\Omega$ is $\mathfrak{g}$-invariant.

\begin{definition} A {\bf quasitriangular structure} or {\bf classical $r$-matrix} on a Lie bialgebra $\mathfrak{g}$ is an element $r\in \mathfrak{g}\otimes \mathfrak{g}$ such that $\delta(x)=[x\otimes 1+1\otimes x,r]$ for $x\in \mathfrak{g}$, $r+r^{21}=\Omega$ is $\mathfrak{g}$-invariant and $r$ satisfies the classical Yang-Baxter equation. This structure is called {\bf triangular} if $\Omega=0$, i.e., $r$ is skew-symmetric. A Lie bialgebra equipped with a (quasi)triangular structure 
is called a {\bf (quasi)triangular Lie bialgebra.} The quasitriangular Lie bialgebra $(\mathfrak{g},r)$ derived from a quasitriangular quantized universal enveloping algebra $(A,\R)$ is called the {\bf quasiclassical limit} of $(A,\R)$, and $(A,\R)$ is called a {\bf quantization} of $(\mathfrak{g},r)$. 
\end{definition}

\begin{exercise}  Let $\mathfrak{g}$ is a Lie algebra and $r\in \mathfrak{g}\otimes \mathfrak{g}$ an element such that $r+r^{21}=\Omega$ is $\mathfrak{g}$-invariant and $r$ satisfies the classical Yang-Baxter equation. Set $\delta(x):=[x\otimes 1+1\otimes x,r]$ for $x\in \mathfrak{g}$. Show that $\delta$ is a Lie bialgebra structure on $\mathfrak{g}$, so $(\mathfrak{g},r)$ 
is a quasitriangular Lie bialgebra.   
\end{exercise} 

\begin{theorem}\label{EK2} (\cite{EK1}, see also \cite{ES}, Ch. 21)
Any quasitriangular Lie bialgebra $(\mathfrak{g},r)$ admits a quantization
$(A,\R)$ such that $A\simeq U(\mathfrak{g})[[\hbar]]$ as algebras. Moreover, this is achieved by 
extending the functor of Theorem \ref{EK1}. 
\end{theorem} 

\subsection{Poisson-Lie groups} (\cite{ES}, Ch. 2, \cite{CP}, Ch. 1). The notion of a Lie bialgebra is the infinitesimal version of the notion of a Poisson-Lie group. 

\subsubsection{}
\begin{definition} A {\bf Poisson-Lie group} is a Lie group $G$ equipped with Poisson bracket such that multiplication $G\times G \to G$ is Poisson map.
\end{definition} 

The Poisson bracket can be defined by the {\bf Poisson bivector} $\Pi$: 
\begin{equation}
\{f,g\} = (df\otimes dg,\Pi),~~~ \Pi \in \Gamma(G,\Lambda^{2}TG).
\end{equation}
Trivializing $TG$ by right translations, one can view $\Pi$ as a function
\begin{equation}
\label{Pidef}
\Pi: ~ G \to \Lambda^{2} \mathfrak{g}.
\end{equation}
The Jacobi identity for $\lbrace,\rbrace$ then translates to the condition that the {\bf Schouten bracket} of $\Pi$ with itself is zero: $[\Pi,\Pi]=0$ (see \cite{ES}, p.8). Also, the condition that $\Pi$ is a Poisson-Lie structure (i.e., invariant under multiplication) translates into the equation
\begin{equation}
\Pi(xy) = \Pi (x) + x \Pi(y) x^{-1},
\end{equation}
i.e., $\Pi$ is a {\bf 1-cocycle} of $G$ with coefficients in $\wedge^2\mathfrak{g}$. 
Therefore, the map 
$$
\delta = d\Pi: \mathfrak{g}\to \Lambda^2 \mathfrak{g},~ \delta\in Z^1(\mathfrak{g},\Lambda^2 \mathfrak{g})
$$ 
is a Lie bialgebra structure on $\mathfrak{g}$. 

\begin{theorem}\label{dri} (Drinfeld) The functor $G\mapsto \mathrm{Lie}\,G$ is an equivalence between the category of simply connected real or complex Poisson-Lie groups and the category of finite dimensional Lie bialgebras over $\mathbb{R}$ or $\mathbb{C}$.
\end{theorem} 

Note that there are also other, more algebraic flavors of Poisson Lie groups which can be defined over any field of characteristic zero -- {\bf Poisson formal groups} and 
{\bf Poisson algebraic groups}, which are defined similarly and have similar properties, see e.g. \cite{EK2} (the only change is 
that the underlying group is formal or algebraic rather than a Lie group). The only significant difference 
is that in the algebraic case, Theorem \ref{dri} does not hold, as there exist finite dimensional 
Lie algebras which are not Lie algebras of algebraic groups. 

\begin{definition} A {\bf deformation quantization} of an affine Poisson algebraic group $(G,\Pi)$ is a flat formal Hopf algebra deformation $B$ of $\mathcal O(G)$ over $\Bbb C[[\hbar]]$ such that 
the induced Poisson bracket on $\mathcal O(G)$ is given by $\Pi$. 
\end{definition}

Quantizations of Poisson-Lie and Poisson formal groups are defined similarly. 

The following theorem is a dual version of Theorem \ref{EK1}. 

\begin{theorem}\label{EK3}\footnote{Theorems \ref{EK1},\ref{EK2},\ref{EK3} answer questions of V. Drinfeld.} (\cite{EK2}, see also \cite{ES}, Ch. 21) Any Poisson group $G$ (of any flavor) admits a functorial deformation quantization $B=\mathcal O_\hbar(G)$. Moreover, there is a linear isomorphism $B\cong \mathcal O(G)[[\hbar]]$ 
such that for $f,g\in \mathcal O(G)$ their deformed product is given by 
$$
f*g=fg+\hbar C_1(f,g)+\hbar^2 C_2(f,g)+...,
$$
where $C_i$ are bidifferential operators on $G$. 
\end{theorem} 

\subsubsection{}
Suppose now that $(G,\Pi)$ is a Poisson-Lie group, and the underlying 
Lie bialgebra $\mathfrak{g}$ has a quasitriangular structure $r$. 
In this case we have the following proposition. 

\begin{proposition} The Poisson bivector $\Pi$ expresses via $r$ by the formula
\begin{equation}
\Pi = L(r)-R(r),
\end{equation}
 where $L$ and $R$ are left and right translations.
\end{proposition}

\subsection{Manin triples} (\cite{CP}, 1.3B, \cite{ES}, Ch. 4). 

\subsubsection{} A natural source of Lie bialgebras is the notion of a Manin triple.
 
\begin{definition} A {\bf Manin triple} is a triple of Lie algebras $(\mathfrak{g},\mathfrak{g}_+,\mathfrak{g}_-)$ such that $\mathfrak{g} = \mathfrak{g}_+ \oplus \mathfrak{g}_-$ as vector spaces and $\mathfrak{g}$ is equipped with an invariant symmetric bilinear form which identifies 
$\mathfrak{g}_-$ with $\mathfrak{g}_+^*$ such that $\mathfrak{g}_+,\mathfrak{g}_-$ are isotropic.
\end{definition} 

Note that in this case $\mathfrak{g}_+,\mathfrak{g}_-$ are necessarily Lagrangian (i.e., maximal isotropic). 

If $(\mathfrak{g},\mathfrak{g}_+,\mathfrak{g}_-)$ is a Manin triple then $\mathfrak{g}_+^* \simeq \mathfrak{g}_-$, so $\mathfrak{g}_+$ has a natural Lie coalgebra structure $\delta$. 

\begin{exercise} Show that $(\mathfrak{g}_+,\delta)$ is a Lie bialgebra. 
\end{exercise} 

\subsubsection{} Thus, every Manin triple defines a Lie bialgebra. It turns out that 
one can also go back and recover a Manin triple from a Lie bialgebra.
This construction is a classical analog of the Drinfeld quantum double construction, 
and is called {\bf the classical Drinfeld double construction}. 

Namely, let $\mathfrak{g}_{+}$ be a Lie bialgebra, and $\mathfrak{g}_- := \mathfrak{g}_+^*$. Recall 
that the Lie cobracket of  $\mathfrak{g}_{+}$ makes $\mathfrak{g}_-$ into a Lie algebra. 
Define $\D(\mathfrak{g}_{+}) := \mathfrak{g} = \mathfrak{g}_+ \oplus \mathfrak{g}_-$, and let 
$(\cdot,\cdot)$ be the symmetric bilinear form on $\mathfrak{g}$ 
which arises from the natural pairing of $\mathfrak{g}_+$ with $\mathfrak{g}_-$.

\begin{exercise} 
Show that there exists a unique Lie bracket on $\mathfrak{g}$ extending the bracket on $\mathfrak{g}_+,\mathfrak{g}_-$ such that $(\cdot,\cdot)$ is invariant, and 
$(\mathfrak{g},\mathfrak{g}_+,\mathfrak{g}_-)$ is a Manin triple. 
Namely, this extension of the bracket is defined by the formula
\begin{equation}
[a, b] = {\rm ad}^*_{a}(b) - {\rm ad}^*_{b}(a),\ a\in \mathfrak g_+,b\in \mathfrak g_-.
\end{equation}
\end{exercise} 

The Lie algebra $\mathfrak{g}$ is called the {\bf classical Drinfeld double} of the Lie bialgebra $\mathfrak{g}_+$. If $\mathfrak{g}_+$ is finite dimensional, then $\mathfrak{g}$ is a Lie bialgebra itself, with cobracket defined by the condition that $\mathfrak{g}_+$ and $\mathfrak{g}_-^{\cop}$ are Lie subbialgebras, i.e., 
$$
\delta_{\mathfrak{g}} = \delta_{{\mathfrak{g}}_{+}} - \delta_{{\mathfrak{g}}_{-}}.
$$ 

\begin{exercise} Show that if $\mathfrak{g}_+$ is a finite dimensional Lie bialgebra and 
$\mathfrak{g}$ its Drinfeld double then 
\begin{equation}
\D(\mathfrak{g}) \cong \mathfrak{g} \oplus \mathfrak{g},
\end{equation}
with cobracket 
\begin{equation} 
\delta_{\D(\mathfrak{g})} = (\delta_{\mathfrak{g}},-\delta_{\mathfrak{g}}).
\end{equation}
\end{exercise} 

\subsubsection{} 
\begin{exercise} Show that the Lie bialgebra $\mathfrak{g}=\D(\mathfrak{g}_+)$ is quasi-triangular, with $r=\sum_i a_i \otimes a_i^*$, where $\lbrace a_i\rbrace,\lbrace a^*_i\rbrace$ are dual bases of $\mathfrak{g}_+$ and $\mathfrak{g}_-$.
\end{exercise} 

\begin{example} Let $\mathfrak{g}$ be a simple Lie algebra, $\mathfrak{b_\pm}\subset \mathfrak{g}$ -- a pair of opposite Borel subalgebras, $\mathfrak{h}=\mathfrak{b}_+\cap \mathfrak{b}_-$ the corresponding Cartan subalgebra. Let $\phi_\pm: \mathfrak{b}_\pm\to \mathfrak{h}$ be the projections. 
Then a Manin triple can be constructed as follows: 
\begin{equation}
\widetilde{\mathfrak{g}} = \mathfrak{g} \oplus \mathfrak{h} = \widetilde{\mathfrak{b}}_+\oplus \widetilde{\mathfrak{b}}_-,
~~~
\widetilde{\mathfrak{b}}_+ =  (\id,\phi_+)(\mathfrak{b}_+),
~~~
\widetilde{\mathfrak{b}}_- =   (\id,-\phi_-)(\mathfrak{b}_-).
\end{equation}
Indeed, $\widetilde{\mathfrak{g}}$ carries an invariant symmetric inner product $(\cdot,\cdot)_{\widetilde{\mathfrak{g}}} = (\cdot,\cdot)_{\mathfrak{g}} - (\cdot,\cdot)_{\mathfrak{h}}$, and 
the Lie subalgebras $\widetilde{\mathfrak{b}}_+,\widetilde{\mathfrak{b}}_-\subset \widetilde{\mathfrak{g}}$ are isotropic with respect to this inner product. This gives rise to a  Lie bialgebra structure on $\widetilde{\mathfrak{b}}_+$ as well as on $\widetilde{\mathfrak{g}} = \D(\mathfrak{b}_+)$. Moreover, the cobracket vanishes on the second summand $\mathfrak{h}$, so we can consider the quotient $(\mathfrak{g}\oplus \mathfrak{h})/\mathfrak{h}= \mathfrak{g}$, and get a quasitriangular structure on $\mathfrak{g}$ given by
\begin{equation}
r = \dfrac{1}{2}\sum_i x_{i}\otimes x_{i} + \sum_{\alpha>0} e_{\alpha}\otimes f_{\alpha}, 
\end{equation} 
where $\lbrace x_i\rbrace$ is an orthonormal basis of $\mathfrak{h}$ and 
$e_\alpha,f_\alpha$ are root elements such that $(e_\alpha,f_\alpha)=1$. 
Note that 
\begin{equation} \Omega = r + r^{21} = \sum_i x_{i}\otimes x_{i} + \sum_{\alpha>0} e_{\alpha}\otimes f_{\alpha},
 + \sum_{\alpha>0} f_{\alpha}\otimes e_{\alpha} 
\end{equation}
is the Casimir tensor of $\mathfrak{g}$.

\subsubsection{}
These results extend to the case when $\mathfrak{g}$ is any
symmetrizable Kac-Moody algebra, but when $\mathfrak{g}$ is infinite dimensional, 
one needs to be careful to take graded duals (instead of full dual). So in this generality $\mathfrak{g}, \mathfrak{b}_+, \mathfrak{b}_-$ are still Lie bialgebras, 
with cobracket defined on generators by 
\begin{equation}
\delta(e_i) = d_ie_i\wedge h_i,~~~
\delta(f_i) = d_if_i\wedge h_i,~~~
\delta(h_i) = 0
\end{equation}
where $d_i$ are the symmetrizing numbers for the generalized Cartan matrix, 
but $\mathfrak{g}$ is not quasitriangular in the literal sense (as $r$ is given by an infinite sum).
\end{example} 

\subsubsection{} An important example of an infinite dimensional Lie bialgebra is the 
{\bf Yangian Lie bialgebra}. To define it, fix a simple finite dimensional Lie algebra $\overline{\mathfrak{g}}$. Then we can define a Manin triple 
\begin{equation}
\mathfrak{g} = \overline{\mathfrak{g}}((t^{-1})),
~~~
\mathfrak{g}_+ = \overline{\mathfrak{g}}[t],
~~~
\mathfrak{g}_- = t^{-1}\overline{\mathfrak{g}}[[t^{-1}]],
~~~\\
(f,g)_{\mathfrak{g}} = \mathrm{Res}\, (f,dg)_{\overline{\mathfrak{g}}}.
\end{equation}

As $\mathfrak{g}_- \simeq \mathfrak{g}_+^*$ there is a Lie bialgebra structure on $\mathfrak{g}_+$ and on $\mathfrak{g}$.  If $a_i$ and $a_i^*$ are dual bases of $\overline{\mathfrak{g}}$ and $\overline{\mathfrak{g}}^*$, then the $r$-matrix is
\begin{equation}
r = \sum_{i}\sum_{n=0}^\infty a_i t^{n}\otimes a_i^* u^{-n-1} =  \Omega \left(\dfrac{1}{u}+\dfrac{t}{u^2}+... \right)= \dfrac{\Omega}{u-t},~~~ \Omega = \sum_i a_i\otimes a_i^*.
\end{equation}
So the classical Yang-Baxter equation for $r$ may be written as the equation
\begin{equation}
[r^{12}(z_1-z_2),r^{13}(z_1-z_3)]+[r^{12}(z_1-z_2),r^{23}(z_2-z_3)]+[r^{13}(z_1-z_3),r^{23}(z_2-z_3)]=0,
\end{equation}
for the function $r(z) = \frac{\Omega}{z}$. 
This equation is called the {\bf classical Yang-Baxter equation with spectral parameter}, and its solution
$r(z)$ is called {\bf Yang's classical $r$-matrix with spectral parameter}. 

Note that we seem to have
\begin{equation}
r(t-u) + r^{21}(u-t) = \dfrac{\Omega}{u-t} + \dfrac{\Omega}{t-u} = 0,
\end{equation}
so at first sight $r$ looks like a triangular structure. However, this is not really so, since 
we expand $r(t-u)$ and $r^{21}(u-t)$ in the regions $|t|<|u|$ and $|t|>|u|$ respectively, 
which don't intersect with each other, and 
\begin{equation}
\left(\dfrac{1}{u-t}\right)_{|t|<|u|} + \left(\dfrac{1}{t-u}\right)_{|t|>|u|} = 
\sum_{n\in \mathbb{Z}}\dfrac{t^n}{u^{n+1}} = t^{-1}\delta(u/t),
\end{equation}
where $\delta(z)=\sum_{n\in \Bbb Z}z^n$ is a non-zero series defining a non-zero distribution on the circle $|z|=1$, namely, $\delta_1$ (even though this distribution is zero outside $1$).   
Drinfeld called such a structure {\bf pseudotriangular}. 

\subsection{Representations theory of $U_q(\mathfrak{g})$} (\cite{J}, Ch. 5, \cite{Lu}, Ch. 6, \cite{CP}, Ch. 10).
Let $\mathfrak{g}$ be a symmetrizable Kac-Moody algebra (for example, finite dimensional or affine), and $\mathfrak{h}$ be the Cartan subalgebra of $\mathfrak{g}$. 

\begin{definition} The {\bf category $\mathcal{O}$} of representations of $\mathfrak{g}$ 
is the category of $\mathfrak{g}$-modules $V$ such that the 
action of $\mathfrak{h}$ is diagonalizable with integer eigenvalues of $h_i$ 
and finite dimensional eigenspaces, 
and all weights fit into finitely many sets of the form $\lambda-Q_+$, where $Q_+$ 
is spanned over $\Bbb Z_+$ by the simple positive roots $\alpha_i$. 
The subcategory of {\bf integrable representations} $\O_{\mathrm{int}} \subset \O$ consists of representations that are locally finite with respect to each $\mathfrak{sl}_2$ subalgebra $\langle e_i,f_i,h_i \rangle$, i.e., such that for every vector $v$ the space 
$U(\mathfrak{sl}_2)v$ is finite dimensional. 
\end{definition}

One can make a similar definition for the quantum group $U_q(\mathfrak{g})$ where 
$q\in \Bbb C^\times$ is not a root of unity. Denote the corresponding categories $\mathcal O_q$,
$\mathcal O_{\mathrm{int},q}\subset \mathcal O_q$. So the elements $K_i$ act on every object 
of $\mathcal O_q$ diagonalizably with eigenvalues being integer powers of $q$ (i.e., these are representations of type I for each simple root subalgebra). 

\begin{theorem}\label{kacth} (V. Kac, see \cite{Ka}) $\O_{\mathrm{int}}$ is semisimple and its simple objects are $L_{\lambda}$, $\lambda\in P_+$, where $P_+$ is the set of dominant integral weights. Moreover, the character of $L_\lambda$ is given by the Weyl-Kac character formula.
\end{theorem} 

\begin{theorem} (G. Lusztig, \cite{Lu}) Theorem \ref{kacth} also holds for the category $\O_{\mathrm{int},q}$.
\end{theorem} 

The simple objects of category $\O$ for $\mathfrak{g}$ 
are the simple modules $L_{\lambda}$ which are irreducible 
quotients of Verma modules $M_{\lambda}$ with highest weight $\lambda\in P$ (the weight lattice). 
These modules are easy to define and classify, but hard to study, in particular their characters 
are very complicated and expressed through {\bf Kazhdan-Lusztig polynomials}. However, 
things do not get any worse for quantum groups, as long as $q$ is generic: 

\begin{theorem} (\cite{EKa})\footnote{This theorem answers a question of V. Drinfeld and I. M. Gelfand.}  For generic $q$ the characters of the irreducible modules 
with highest weight $\lambda$ for $\mathfrak{g}$ and $U_q(\mathfrak{g})$
are the same: ${\rm ch} L_{\lambda} = {\rm ch} L^q_{\lambda}$.
\end{theorem} 

\begin{theorem}\label{dete} If $\mathfrak{g}$ is finite dimensional then the category $\Rep_{\mathrm{f.d.}} U_q(\mathfrak{g})$ of finite dimensional representations of $U_q(\mathfrak{g})$, viewed as a monoidal category, determines $q$ up to $q\mapsto q^{-1}$. As a braided monoidal category it determines $q$ uniquely. 
\end{theorem} 

\begin{remark} The second statement of Theorem \ref{dete} is easier since $q$ can be recovered from the eigenvalues of the squared braiding. But the first statement shows that up to inversion $q$ can be recovered from the associativity isomorphism alone. For instance, if $\mathfrak{g}=\mathfrak{sl}_2$ and $V$ 
the 2-dimensional tautological representation of $U_q(\mathfrak{g})$, then there exists an isomorphism 
$\phi: V\to V^*$, and one can compute the quantum trace of the isomorphism 
$(\phi^*)^{-1}\phi: V\to V^{**}$ (which can be defined just from the monoidal structure). 
This trace is clearly independent on the choice of $\phi$ and equals $-q-q^{-1}$. 
A similar argument can be used for other simple Lie algebras $\mathfrak{g}$.  
\end{remark}

\section{Affine quantum groups} The goal of this section is to discuss affine quantum groups and their finite dimensional representations. The theory of such representations is very rich and has applications in several areas of mathematics and mathematical physics. Before doing so, however, let us discuss 
finite dimensional representations of classical affine algebras, which are significantly simpler.  

\subsection{Finite dimensional representations of $\widehat{\mathfrak{g}}$} (\cite{Ch}).
Let $\widehat{\mathfrak{g}}$ be an affine Lie algebra, $\widehat{\mathfrak{g}} = \mathfrak{g}[t,t^{-1}]\oplus \mathbb{C}{\rm k}$, where $\mathfrak{g}$ is a simple Lie algebra and ${\rm k}$ is the central element (\cite{Ka}). The commutator in this Lie algebra has the form 
$$
[a(t),b(t)] = [a,b](t) + \mathrm{Res}\,(da,b)\cdot {\rm k}.
$$
\begin{proposition} 
In finite dimensional representations of $\widehat{\mathfrak{g}}$ we have ${\rm k}=0$.
\end{proposition}

\begin{proof} We have 
 $[ht,ht^{-1}]=2{\rm k}$, i.e., these elements generate a 3-dimensional Heisenberg Lie algebra. In any finite dimensional representation of the Heisenberg algebra, ${\rm k}$ has to be nilpotent. But ${\rm k}=\sum_i k_i h_i$ is semisimple, as all the generators $h_i$ are semisimple.
\end{proof} 
 
Now define the {\bf evaluation homomorphism} $\phi_z:\, \widehat{\mathfrak{g}} \to \mathfrak{g},~ \phi_z(a(t)) = a(z),~\phi({\rm k})=0$, where $ z\in \mathbb{C}^\times $. For a representation $V$ 
of $\mathfrak{g}$, define the {\bf evaluation representation} 
$V(z):=\phi^*V$. 

The classification of finite dimensional irreducible representations of $\widehat{\mathfrak{g}}$ 
is given by the following theorem. 

\begin{theorem}
(i) If $V_1,...,V_n$ are non-trivial irreducible finite dimensional representations of $\mathfrak{g}$ and $z_1,...,z_n\in \Bbb C^\times$ are distinct then the representation $V_1(z_1)\otimes ... \otimes V_n(z_n)$ is irreducible.

(ii) Any finite dimensional irreducible representation of $\widehat{\mathfrak{g}}$ is of this form uniquely up to permutation.
\end{theorem} 

The following exercise explains why $z_i$ in this theorem need to be distinct. 

\begin{exercise} Let $V,W$ be nontrivial irreducible representations of $\mathfrak{g}$. Show that $V\otimes W$ is reducible, and so $V(z)\otimes W(z) = (V\otimes W)(z)$ is reducible. 

{\bf Hint:} Write $\Hom_{\mathfrak{g}}(V\otimes W,V\otimes W)$ as $\Hom_{\mathfrak{g}}(V\otimes V^*,W\otimes W^*)$ and then 
use that if $V$ is nontrivial then $V\otimes V^*$ contains $\Bbb C$ and $\mathfrak{g}$. 
\end{exercise} 

\subsection{The quantum group $U_q(\widehat{\mathfrak{sl}}_2)$ and its representations} (\cite{CP}, 12.2, \cite{CP2}).

\subsubsection{} The simplest affine quantum group is the {\bf quantum affine algebra} $U_q(\widehat{\mathfrak{sl}}_2)$, which has generators $\langle e_i,f_i,K_i^{\pm}\rangle_{i=0,1}$ and defining relations 
\begin{equation}
K_ie_i=q^2e_iK_i,\ K_if_i=q^{-2}f_iK_i,\ [e_i,f_i]=\frac{K_i-K_i^{-1}}{q-q^{-1}},\ i=0,1; 
\end{equation}
\begin{equation}
K_0 K_1 = K_1 K_0,~~~
[e_0,f_1] = 0,~~~
[e_1,f_0] = 0;
\end{equation}
\begin{equation}
K_0 e_1 K_0^{-1} = q^{-2}e_1,~~~
K_1 e_0 K_1^{-1} = q^{-2}e_0,~~~
K_0 f_1 K_0^{-1} = q^{2}f_1,~~~
K_1 f_0 K_1^{-1} = q^{2}f_0;
\end{equation}
\begin{equation}
e_i^3 e_j - (q^2+1+q^{-2})e_i^2 e_j e_i + (q^2+1+q^{-2})e_i e_j e_i^2 - e_j e_i^3 = 0,\ i\ne j;
\end{equation}
\begin{equation}
f_i^3 f_j - (q^2+1+q^{-2})f_i^2 f_j f_i + (q^2+1+q^{-2})f_i f_j f_i^2 - f_j f_i^3 = 0,\ i\ne j,
\end{equation}
where the last two sets of relations are the Serre relations (here $q\in \Bbb C^\times$ is not a root of unity). This is a Hopf algebra with coproduct, counit and antipode defined as usual for each individual quantum $\mathfrak{sl}_2$-subalgebra generated by $e_i,f_i,K_i$ for $i=0,1$.  Thus the element $\bold K := K_0 K_1$ is central. 

For an affine Lie algebra $\mathfrak{g}$, we say that a finite dimensional representation $Y$ is of {\bf type I} if so is its restriction to every root $\mathfrak{sl}_2$-subalgebra 
generated by $e_i,f_i,h_i$ (see Problem \ref{typeI} in Subsection \ref{prob1}). 

\begin{exercise}  Show that $\bold K$ acts by $1$ in all finite dimensional type I representations of $U_q(\widehat{\mathfrak{sl}}_2)$.  
\end{exercise} 

The quantum group $U_q(\widehat{\mathfrak{sl}}_2)$ is the quantized Kac-Moody algebra with Cartan matrix 
\begin{equation}
A=
\left(
\begin{array}{cc}
2 & -2\\
-2 & 2
\end{array}
\right).
\end{equation}

\begin{remark} In the classical limit $q\to 1$ the Hopf algebra $U_q(\widehat{\mathfrak{sl}}_2)$ degenerates into 
the universal enveloping algebra $U(\widehat{\mathfrak{sl}}_2)$ of the Lie algebra $\widehat{\mathfrak{sl}}_2=\mathfrak{sl}_2\otimes \Bbb C[t,t^{-1}]\oplus \Bbb C{\rm k}$, the affine Kac-Moody algebra attached to the same Cartan matrix. In this limit the generators $e_i,f_i,h_i$ specialize as follows: 
\begin{equation}
e_1 = e, ~~~ f_1 = f, ~~~ h_1 = h, ~~~ e_0 = f\cdot t, ~~~ f_0 = e\cdot t^{-1},~~~ h_0 = {\rm k}-h,
\end{equation}
where $K_i=q^{h_i}$ and $\bold K=q^{\rm k}$. 
\end{remark} 

\subsubsection{}
\begin{definition} (i) The {\bf evaluation homomorphism} $\phi_z:\, U_q(\widehat{\mathfrak{sl}}_2) \to U_q(\mathfrak{sl}_2)$ is the composition $\phi_z := \varphi\circ \tau_z$, where 
\begin{equation}
\varphi(e_1) = \varphi(f_0) = e, ~~~
\varphi(f_1) = \varphi(e_0) = f, ~~~
\varphi(K_1) = \varphi(K_0^{-1}) = K
\end{equation}
and $\tau_z$ is an automorphism of $U_q(\widehat{\mathfrak{sl}}_2)$ defined by 
\begin{equation}
\tau_z(e_0) = z e_0,~~~
\tau_z(f_0) = z^{-1} f_0,~~~
\text{and $\tau_z=\id$ on the other generators.}
\end{equation}

(ii) Given a representation $Y$ of $U_q(\mathfrak{sl}_2)$, the corresponding {\bf evaluation representation} $Y(z)$ of $U_q(\widehat{\mathfrak{sl}}_2)$ corresponding to $z\in \Bbb C^\times$
is the representation of $U_q(\widehat{\mathfrak{sl}}_2)$ on the space $Y$ given by the formula
$\pi_{Y(z)}(a) = \pi_Y(\phi_z(a))$ for all $a\in U_q(\widehat{\mathfrak{sl}}_2)$. More generally, 
given a representation $Y$ of $U_q(\widehat{\mathfrak{sl}}_2)$, we may define 
its twist $Y(z)$ by $\pi_{Y(z)}=\pi_Y\circ \tau_z$. 
\end{definition} 

Note that $Y(z)(w)=Y(zw)$, $(X\otimes Y)(z)=X(z)\otimes Y(z)$, $Y(z)^*=Y^*(z)$ 
for any representations $X,Y$ of $U_q(\widehat{\mathfrak{sl}}_2)$.

We warn the reader, however, that unlike the classical case $q=1$, 
$\phi_z$ is {\bf not} a Hopf algebra homomorphism, 
and as a result for representations $W,U$ of $U_q({\mathfrak{sl}}_2)$, $(W\otimes U)(z)$ 
is typically {\bf not} isomorphic to $W(z)\otimes U(z)$ (as demonstrated, for instance, by the examples below). In particular, $W\mapsto W(z)$
is not a monoidal functor from the category of $U_q({\mathfrak{sl}}_2)$-modules to the category of  $U_q(\widehat{\mathfrak{sl}}_2)$-modules. 

\begin{example} Let $V= \mathbb{C}^2$ be the 
two-dimensional tautological representation of $U_q(\mathfrak{sl}_2)$. Then 
the representation $V(z)$ is defined by the formulas
\begin{equation}
e_1
\to 
\left(
\begin{array}{cc}
0 & 1\\
0 & 0
\end{array}
\right),
~~~
f_1
\to 
\left(
\begin{array}{cc}
0 & 0\\
1 & 0
\end{array}
\right),
~~~
K_1
\to 
\left(
\begin{array}{cc}
q & 0\\
0 & q^{-1}
\end{array}
\right),
\end{equation}
\begin{equation}
e_0
\to 
\left(
\begin{array}{cc}
0 & 0\\
z & 0
\end{array}
\right),
~~~
f_0
\to 
\left(
\begin{array}{cc}
0 & z^{-1}\\
0 & 0
\end{array}
\right),
~~~
K_0
\to 
\left(
\begin{array}{cc}
q^{-1} & 0\\
0 & q
\end{array}
\right).
\end{equation}
\end{example} 

\begin{exercise}\label{ty1} Check that any nontrivial two-dimensional representation of 
 $U_q(\widehat{\mathfrak{sl}}_2)$ (of type I) is isomorphic to $V(w)$ for a unique $w\in \Bbb C^\times$.
 \end{exercise} 

\subsubsection{}
By Exercise \ref{ty1}, the dual $V(z)^*$ is isomorphic to $V(w)$ for some $w\in \Bbb C^\times$. To find the evaluation parameter $w$, note that in $V(w)$ we have 
\begin{equation}
w = {\rm tr} (e_0 e_1). 
\end{equation}
But 
\begin{equation}
\pi_{V(z)^*}(e_{0}) = \pi_{V(z)}(S(e_{0}))^* = \pi_{V(z)}(-e_{0} K^{-1}_{0})^*,~~~
\pi_{V(z)^*}(e_{1}) = \pi_{V(z)}(S(e_{1}))^* = \pi_{V(z)}(-e_{1} K^{-1}_{1})^*,
\end{equation}
so we find
\begin{equation}
w = {\rm tr}_{V(z)} (e_1 K_1^{-1} e_0 K_0^{-1}) \Rightarrow w = q^{2}z.
\end{equation}
Thus
\begin{equation}
V(z)^* \cong V(q^2 z),~~~
V(z)^{**} \cong V(q^4 z). 
\end{equation}
So we see that in the category of finite dimensional representations
of $U_q(\widehat{\mathfrak{sl}}_2)$ (unlike $U_q(\mathfrak g)$ for finite dimensional $\mathfrak g$), 
taking the double dual is a non-trivial operation on the set of isomorphism classes of irreducible representations. 

\begin{remark} In fact, it is easy to show that the equalities $W(z)^*=W^*(q^2z)$, $W(z)^{**} = W(q^4 z)$ hold for any finite dimensional representation $W$ of $U_q({\mathfrak{sl}}_2)$, in contrast with the equality $Y(z)^*=Y^*(z)$ for a representation $Y$ 
of $U_q(\widehat{\mathfrak{sl}}_2)$.\footnote{This difference arises because, as we have already mentioned, the functor 
$\Rep (U_q(\mathfrak{sl}_2))\to \Rep(U_q(\widehat{\mathfrak{sl}}_2))$ sending $W$ to $W(z)$ is not a monoidal functor.} More generally, for any finite dimensional Lie algebra $\mathfrak g$ and any finite dimensional representation $Y$ of $U_q(\widehat{\mathfrak g})$, $Y^{**} = Y(q^{2h^{\vee}})$, where $h^{\vee}$ is the dual Coxeter number of $\mathfrak g$.
The reason is that the squared antipode $S^2$ for a quantized Kac-Moody algebra $\mathfrak{a}$ has the form $S^2 = \mathrm{Ad}\,(q^{2\rho_{\mathfrak{a}}})$, and 
for $\mathfrak{a}=\widehat{\mathfrak{g}}$ we have 
$\rho_{\widehat{\mathfrak{g}}}=\rho_{\mathfrak{g}}+h^\vee d$, where $d$ is the grading element. 
\end{remark} 

\subsubsection{}  Since $V(z)^*=V(q^2z)$, we have homomorphisms 
\begin{equation}
V(z)\otimes V(q^2 z) = V(z)\otimes V(z)^{*} \xleftarrow{\mathrm{coev}} \mathbb{C},
\end{equation} 
\begin{equation}
V(z)\otimes V(q^{-2} z) = V(zq^{-2})^{*} \otimes V(zq^{-2}) \xrightarrow{\mathrm{ev}} \mathbb{C},
\end{equation}
which show that these tensor products are reducible.

\begin{exercise} 
Show that there is a pair of short exact sequences
\begin{equation}
0 \to \mathbb{C} \to V(z)\otimes V(q^2 z) \to W(qz) \to 0,
\end{equation} 
\begin{equation}
0 \to W(q^{-1}z) \to V(z)\otimes V(q^{-2} z) \to \mathbb{C}\to 0,
\end{equation}
where $W$ is the three dimensional irreducible representation of 
$U_q(\mathfrak{sl}_2)$, and prove that these sequences are {\bf not} split. 
\end{exercise} 

It follows that the category of finite dimensional representations of $U_q(\widehat{\mathfrak{sl}}_2)$ does not admit a braiding, as $X\otimes Y$ is not always isomorphic to $Y\otimes X$. However, as we will see, the Jordan-H\"older series of these two representations always coincide.

\begin{exercise} Show that $V(z)\otimes V(u)$ is irreducible for all $z,u\in \Bbb C^\times$ except when $z/u = q^{\pm 2}$, and $V(z)\otimes V(u) \simeq V(u)\otimes V(z)$ away from these points.
\end{exercise} 

\begin{remark} More generally, we will see that if $X,Y$ are any two irreducible representations 
of $U_q(\widehat{\mathfrak{sl}}_2)$ then the representation $X(z)\otimes Y(u)$
is irreducible except when $z/u$ takes finitely many values. Moreover, this is true and not difficult to check for $U_q(\widehat{\mathfrak{g}})$ for any simple Lie algebra $\mathfrak{g}$. 
\end{remark} 

\subsection{Classification of irreducible finite dimensional representations of 
$U_q(\widehat{\mathfrak{sl}}_2)$}\label{classif} (\cite{CP}, 12.2, \cite{CP2}). 

Let $V_a$ be the irreducible representation of $U_q(\mathfrak{sl}_2)$ with highest weight $a\in \mathbb{Z}_{\geq 0}$. 

\begin{question} When is the representation $V_{a_1}(z_1)\otimes ... \otimes V_{a_n}(z_n)$ of $U_q(\widehat{\mathfrak{sl}}_2)$ irreducible?
\end{question} 

To address this question, we introduce the following definition. 

\begin{definition} A {\bf string} in $\mathbb{C}^\times$ is a finite geometric progression with common ratio $q^2$, i.e. $b,~ q^2 b,~ q^4 b,~...,~q^{2m} b$. Two strings $S,T$ are in {\bf special position} if $S \cup T$ is a string containing $S$ and $T$ properly; otherwise they are in {\bf general position}.
\end{definition} 

Let us now assign to $V_a(z)$ its string $s_{a}(z) := \{ q^{-a+1}z,~...~,q^{a-3}z,q^{a-1}z \}$ of length $a$. For example, for the two-dimensional representation $V(z)$ the string is $s_1(z)=\{ z \}$.

\begin{theorem} (V. Chari -- A. Pressley)
(i) The tensor product $V_{a_1}(z_1)\otimes ... \otimes V_{a_n}(z_n)$ is irreducible $\Leftrightarrow$ the strings $s_{a_i}(z_i)$ are pairwise in general position.

(ii) Such irreducible tensor products are isomorphic $\Leftrightarrow$ they differ by permutation of factors.

(iii) Any irreducible finite dimensional representation of $U_q(\widehat{\mathfrak{sl}}_2)$ (of type $I$) is isomorphic to such a product.
\end{theorem} 

\begin{exercise} Show that any finite multiset in $\mathbb{C}^\times$ can be uniquely presented as a union of strings, pairwise in general position.
\end{exercise} 

\begin{example} Here are some examples of such decompositions: 
\begin{figure}[!h]
\begin{center}
\scalebox{0.8}{
\begin{tikzpicture}
\tikzmath{\r=0.05;};
\draw[fill=black,thick,radius=\r,inner sep=0] (0,0) circle;
\draw[fill=black,thick,radius=\r,inner sep=0] (1,0) circle;
\draw[fill=black,thick,radius=\r,inner sep=0] (2,0) circle;
\draw[fill=black,thick,radius=\r,inner sep=0] (3,0) circle;
\draw[fill=black,thick,radius=\r,inner sep=0] (4,0) circle;

\node at (2,0.5) {$2$};
\node at (3,0.5) {$3$};

\node at (5,0) {$=$};

\draw[fill=black,thick,radius=\r,inner sep=0] (6,0.5) circle;
\draw[fill=black,thick,radius=\r,inner sep=0] (7,0.5) circle;
\draw[fill=black,thick,radius=\r,inner sep=0] (8,0.5) circle;
\draw[fill=black,thick,radius=\r,inner sep=0] (9,0.5) circle;
\draw[fill=black,thick,radius=\r,inner sep=0] (10,0.5) circle;

\draw[fill=black,thick,radius=\r,inner sep=0] (8,0) circle;
\draw[fill=black,thick,radius=\r,inner sep=0] (9,0) circle;

\draw[fill=black,thick,radius=\r,inner sep=0] (9,-0.5) circle;
\end{tikzpicture}
\begin{tikzpicture}
\draw[white] (-2,-0.3)--(-2,0.3);

\tikzmath{\r=0.05;};
\draw[fill=black,thick,radius=\r,inner sep=0] (0,0) circle;
\draw[fill=black,thick,radius=\r,inner sep=0] (1,0) circle;
\draw[fill=black,thick,radius=\r,inner sep=0] (2,0) circle;

\node at (0,0.5) {$2$};
\node at (2,0.5) {$2$};

\node at (3,0) {$=$};

\draw[fill=black,thick,radius=\r,inner sep=0] (4,0.5) circle;
\draw[fill=black,thick,radius=\r,inner sep=0] (5,0.5) circle;
\draw[fill=black,thick,radius=\r,inner sep=0] (6,0.5) circle;

\draw[fill=black,thick,radius=\r,inner sep=0] (4,0) circle;

\draw[fill=black,thick,radius=\r,inner sep=0] (6,-0.5) circle;
\end{tikzpicture}
}
\end{center}
\end{figure}
\end{example} 

So we see that finite dimensional irreducible type I representations $Y$ of $U_q(\widehat{\mathfrak{sl}}_2)$ are parametrized by finite multisets in $\mathbb{C}^\times$. Note that such a multiset is the multiset of roots of a unique polynomial with constant term $1$. This polynomial $P_Y(z)$
is called the {\bf Drinfeld polynomial} of $Y$. For example, $P_{V_a(1)}(z)=(1-q^{-a+1}z)(1-q^{-a+3}z)...(1-q^{a-3}z)(1-q^{a-1}z)$.

\subsection{$R$-matrices with a spectral parameter} (\cite{CP}, Ch. 12, \cite{CP2}, \cite{EFK}, Lecture 9).

\subsubsection{} 
The Hopf algebra $U_q(\widehat{\mathfrak{sl}}_2)/(\bold K-1)$ has a universal $\R$-matrix 
$$
\R = \sum_i a_i \otimes a_i^*
$$ 
coming from the quantum double construction. This is an infinite sum, so if $X,Y$ are finite dimensional representations then $\R|_{X\otimes Y}$ does not make sense, in general. 
However, one can consider the power series
\begin{equation}
(\tau_z\otimes \M1)(\R)|_{X\otimes Y} = R_{X(z)\otimes Y} \in \End(X\otimes Y)[[z]]
\end{equation}
which is called {\bf the R-matrix with spectral parameter}. 

\begin{theorem} (V. Drinfeld) Assume that $|q|<1$. Then this series converges for $|z|<r$ for some $r$. 
If $X,Y$ are irreducible, then 
\begin{equation}
R_{XY}(z) = f_{XY}(z)\overline{R}_{XY}(z)
\end{equation}
where $\overline{R}_{XY}$ is a matrix-valued rational function which does not vanish at any $z\in \Bbb C^\times$ and $f_{XY}$ is a scalar meromorphic function such that $f_{XY}(0)=1$.
\end{theorem}

\subsubsection{} The fact that $X^{**}=X(q^4)$ gives rise to a certain $q$-difference equation on $f_{XY}$, which implies that $f_{XY}$ extends to a meromorphic function on $\Bbb C^\times$. Also, the operator $\overline{R}_{XY}$ satisfies the unitarity condition
\begin{equation}
\overline{R}_{XY}(z)\overline{R}_{YX}^{21}(z^{-1}) = \M1 \otimes \M1, 
\end{equation}
and the operator $P\circ R_{XY}(z/w): X(z)\otimes Y(w)\to Y(w)\otimes X(z)$ 
is an isomorphism away from finitely many values of $z/w$. Finally, 
we have the {\bf hexagon relations} 
$$
R_{X\otimes Y,Z}(z)=R_{XZ}^{13}(z)R_{YZ}^{23}(z),\ 
R_{X,Y\otimes Z}(z)=R_{XZ}^{13}(z)R_{XY}^{12}(z), 
$$
which imply the {\bf quantum Yang-Baxter equation with multiplicative spectral parameter} 
$$
R_{XY}^{12}(z_1/z_2)R_{XZ}^{13}(z_1/z_3)R_{YZ}^{23}(z_2/z_3)=R_{YZ}^{23}(z_2/z_3)R_{XZ}^{13}(z_1/z_3)R_{XY}^{12}(z_1/z_2).
$$
Thus $R$ almost defines a braiding on the category of finite dimensional representations 
(except not quite since $R(z)$ and $R(z)^{-1}$ have poles, and in fact we know that a true braiding does not exist since sometimes $X\otimes Y\ncong Y\otimes X$). This ``almost braiding" is called 
a {\bf meromorphic braided structure}. 

\begin{example} Let $V=V(1)= \langle v_{+},v_{-} \rangle$ be the tautological $2$-dimensional representation. Then 
\begin{equation}
\label{Rmat4by4}
\def\arraystretch{1.5}
\overline R_{VV}(z) =
\left(
\begin{array}{cccc}
1 & 0 & 0 & 0 \\
0 & \dfrac{q(z-1)}{z-q^{2}} & \dfrac{1-q^2}{z-q^{2}} & 0 \\
0 & \dfrac{z(1-q^2)}{z-q^{2}} & \dfrac{q(z-1)}{z-q^{2}} & 0 \\
0 & 0 & 0 & 1
\end{array}
\right)
\end{equation}
where the $R$-matrix is written in the basis  $\langle v_{+}\otimes v_{+}, v_{+}\otimes v_{-}, v_{-}\otimes v_{+} , v_{-}\otimes v_{-} \rangle$ of $\mathbb{C}^2 \otimes \mathbb{C}^2$. We see that 
$\overline R$ has a pole at $z=q^{2}$, and $\overline R^{-1}$ has a pole at $z=q^{-2}$.

If we make the change of variable $z=e^{u}$, we will obtain 
a solution $\widetilde R(u):=\overline{R}(e^{u})$ of the {\bf quantum Yang-Baxter equation with additive spectral parameter} 
$$
\widetilde R_{XY}^{12}(u_1-u_2)\widetilde R_{XZ}^{13}(u_1-u_3)\widetilde R_{YZ}^{23}(u_2-u_3)=\widetilde R_{YZ}^{23}(u_2-u_3)\widetilde R_{XZ}^{13}(u_1-u_3)\widetilde R_{XY}^{12}(u_1-u_2).
$$
Since $\widetilde R(u)$ is a rational function in $e^{u}$, it is called a {\bf trigonometric R-matrix} (as it is a trigonometric function of $iu$). 
\end{example} 

\subsubsection{}
\begin{proposition}\label{irre} If $X(w)\otimes Y$ is irreducible then $\overline{R}_{XY}(w)$ is well-defined and invertible. 
\end{proposition} 

\begin{proof} Suppose the order of pole of $\overline{R}_{XY}(z)$ at $z=w$ is $m$. Since $\overline{R}_{XY}(z)$ is nonvanishing, $m\ge 0$. The map $\lim_{z\to w}(z-w)^mP\circ \overline{R}_{XY}(z)$ 
is a nonzero homomorphism of representations $X(w)\otimes Y\to Y\otimes X(w)$, so by Schur's lemma it is an isomorphism. So by unitarity $\overline R_{YX}(z)$ has a zero of order $m$ at $w^{-1}$. Since $\overline R$ is non-vanishing, we conclude that $m=0$, as claimed.   
\end{proof} 

\begin{corollary} The representation $V_{a_1}(z_1)\otimes...\otimes V_{a_n}(z_n)$ is irreducible if and only if all $R$-matrices 
$$
\overline R_{ij}(z_i/z_j):\, V_{a_i}(z_i)\otimes V_{a_j}(z_j) \to V_{a_i}(z_i)\otimes V_{a_j}(z_j)
$$ 
are defined and invertible. In this case we have isomorphisms
\begin{equation}
P\circ \overline R_{ij}(z_i/z_j):\, V_{a_i}(z_i)\otimes V_{a_j}(z_j) \to V_{a_j}(z_j) \otimes V_{a_i}(z_i).
\end{equation}
\end{corollary} 

\begin{proof} This follows from Proposition \ref{irre} and the description of irreducible representations in terms of strings. 
\end{proof} 

\subsection{Affine quantum groups of higher rank} (\cite{CP}, 12.2, \cite{EFK}, Lecture 9). 

\subsubsection{}
Let us now briefly discuss quantum affine algebras $U_q(\widehat{\mathfrak{g}})$ 
of higher rank. The quantum affine algebra 
$U_q(\widehat{\mathfrak{g}})$ is a Hopf algebra with Kac-Moody generators 
$e_i,f_i,h_i$, $i=0,...,r$, where these generators for $i>0$ generate 
the ``finite-dimensional" quantum group $U_q(\mathfrak{g})$.
We still have the 1-parameter group of automorphisms 
$\tau_z$ of $U_q(\widehat{\mathfrak{g}})$ defined in the same way as in the $\mathfrak{sl}_2$ case: 
$\tau_z(e_0)=ze_0$, $\tau_z(f_0)=z^{-1}f_0$, and $\tau_z$ preserves all other generators. 
Thus for every representation $Y$ of $U_q(\widehat{\mathfrak{g}})$ we can still define the twist 
$Y(z)$ by $\pi_{Y(z)}(a):=\pi_Y(\tau_z(a))$. 

\subsubsection{} So how to construct finite dimensional representations of $U_q(\widehat{\mathfrak{g}})$? First consider the Hopf algebra $U_q(\widehat{\mathfrak{sl}}_n)$. 
This Hopf algebra still has an evaluation homomorphism 
$\varphi:\, U_q(\widehat{\mathfrak{sl}}_n) \to U_q(\mathfrak{sl}_n)$, defined 
to be the identity on $U_q(\mathfrak{sl}_n)$ and 
by the formula $e_0 \mapsto f_{\theta}$, $f_0\mapsto e_\theta$ on the affine root generators, where 
$\theta=(1,0,...,0,-1)$ is the highest root and $f_\theta=E_{n,1}$ is obtained from the recursion 
\begin{equation}
E_{i,i-1}=f_i,~~~
E_{i,i-m} = [E_{i,i-1},E_{i-1,i-m}]_{q},\ [x,y]_q = xy - q yx
\end{equation}
(and similarly for $e_\theta=E_{1,n}$). Thus as before, we can define $\phi_z=\varphi\circ \tau_z$, and 
for any representation $W$ of $U_q(\mathfrak{sl}_n)$ we can define the pull-back $\phi_z^*W = W(z)$, consider tensor products of these and take subquotients. This 
allows us to construct many interesting finite dimensional representations. 

However, it is no longer true for $n>2$ that all irreducible representations are such tensor products, and  things get much more complicated. It is only true that any irreducible is a quotient of such tensor product. The ultimate understanding of the structure of irreducible representations of $U_q(\widehat{\mathfrak{sl}}_n)$ (and general $U_q(\widehat{\mathfrak{g}})$) came only from the works of Varagnolo, Vasserot and Nakajima on equivariant $K$-theory of Nakajima quiver varieties.

\subsubsection{} For $\mathfrak{g}\ne \mathfrak{sl}_n$, 
things get worse --- there is no evaluation homomorphism:
\begin{equation}
\begin{array}{ccc}
U_q (\widehat{\mathfrak{g}}) & \vspace{0.1cm} \xrightarrow{\not\mathrm{ev}} & U_q (\mathfrak{g})\\
\hspace{0.5cm}\nwarrow & & \nearrow \id \hspace{0.5cm} \\ 
& U_q (\mathfrak{g})  & 
\end{array}
\end{equation}
For example, as was shown by Drinfeld, the adjoint representation $\mathfrak{g}$ does not lift to the quantum affine algebra, although $\mathfrak{g}\oplus \mathbb{C}$ does. In more general cases, to extend to $U_q (\widehat{\mathfrak{g}})$ a finite dimensional representation $V_\lambda$ of $U_q(\mathfrak{g})$ with highest weight $\lambda$, one has to add lots of lower weight representations:
\begin{equation}
\widehat{V}_{\lambda} = V_{\lambda} \oplus\bigoplus_{\mu < \lambda} c_{\mu \lambda} V_{\mu}.
\end{equation}
This is best understood through the geometric approach of Varagnolo, Vasserot and Nakajima. 

\subsubsection{}
\begin{proposition} The Grothendieck ring of the category of finite dimensional 
representations of $U_q(\widehat{\mathfrak{g}})$ is commutative.
\end{proposition} 

\begin{proof} Let $M_0,N_0$ be finite dimensional modules 
over a $\Bbb C$-algebra $A$, and $M,N$ their flat formal deformations 
over $\Bbb C[[t]]$. Suppose $M[t^{-1}]\cong N[t^{-1}]$ as $A((t))$-modules. 
Then $M_0$ and $N_0$ have the same composition series (this follows from the existence of the {\bf Jantzen filtrations} on the left $A$-module $M_0$ and the right $A$-module $N_0^*$ such that ${\rm gr}(M_0)\cong {\rm gr}(N_0^*)^*$). 

Now take $z=1+t$ for a formal parameter $t$, 
$M=X(z)\otimes Y$ and $N=Y\otimes X(z)$, where $X,Y$ are finite dimensional 
representations of $U_q(\widehat{\mathfrak{g}})$.  
We have seen above that the R-matrix with spectral parameter defines an isomorphism $M[t^{-1}]\to N[t^{-1}]$. Thus $M_0=X\otimes Y$ and $N_0=Y\otimes X$ have the same composition series, 
which implies the statement. 
\end{proof} 

\subsection{The loop polarization} (\cite{CP}, 12.2, \cite{CP2}). 
Our next goal is to parametrize irreducible finite-dimensional representations of $U_q(\widehat{\mathfrak{g}})$. Usually, finite dimensional representations are labeled by highest weights. However, finite-dimensional representations of $U_q(\widehat{\mathfrak{g}})$, unlike category $\mathcal{O}$ representations, do not have a locally 
nilpotent action of $U_q(\widehat{\mathfrak{n}}_{+})$ generated by $e_i$, $i=0,...,r$, so they don't have highest weights in the usual sense. However, these representations will have highest weights if we consider a different polarization of $\widehat{\mathfrak{g}}$ -- the {\bf loop polarization}. 
This polarization is depicted on the right of the figure below (with the usual polarization 
on the left for comparison). 

For example, consider the case $q=1$, i.e., finite dimensional representations of $\mathfrak{g}[t,t^{-1}]$. Then we have the {\bf loop polarization}
$$
\mathfrak{g}[t,t^{-1}]=\widetilde{\mathfrak{n}}_+\oplus \widetilde{\mathfrak{h}}\oplus \widetilde{\mathfrak{n}}_-,
$$
where 
$$
\widetilde{\mathfrak{n}}_+=
\mathfrak{n}_+[t,t^{-1}],\ \widetilde{\mathfrak{h}}=\mathfrak{h}[t,t^{-1}],\ \widetilde{\mathfrak{n}}_-= \mathfrak{n}_-[t,t^{-1}],
$$
and every finite dimensional representation $V$ of $\mathfrak{g}[t,t^{-1}]$ 
has a highest weight vector $v$ annihilated by $\widetilde{\mathfrak{n}}_+$ (unique up to scaling if $V$ is irreducible). The weight of this vector is then a character of the ``loop Cartan subalgebra" 
$\widetilde{\mathfrak{h}}$, and one can, in fact, decompose the whole representation $V$ in a direct sum of generalized eigenspaces of $\widetilde{\mathfrak{h}}$.  

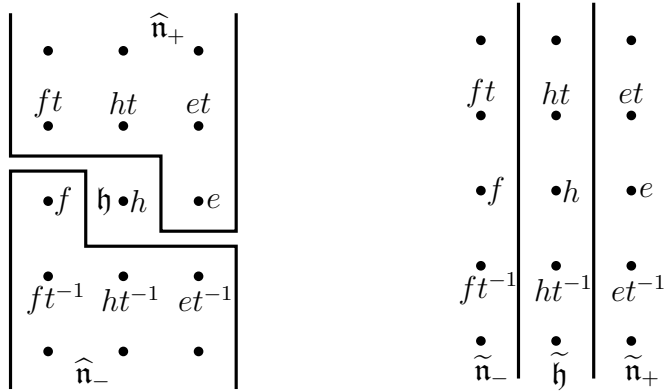
\begin{figure}[h!]
\begin{center}
\scalebox{1}{
\begin{tikzpicture}
\tikzmath{\r=0.05;};

\draw[fill=black,thick,radius=\r,inner sep=0] (0,-2) circle;
\draw[fill=black,thick,radius=\r,inner sep=0] (1,-2) circle;
\draw[fill=black,thick,radius=\r,inner sep=0] (2,-2) circle;
\draw[fill=black,thick,radius=\r,inner sep=0] (0,-1) circle;
\draw[fill=black,thick,radius=\r,inner sep=0] (1,-1) circle;
\draw[fill=black,thick,radius=\r,inner sep=0] (2,-1) circle;
\draw[fill=black,thick,radius=\r,inner sep=0] (0,0) circle;
\draw[fill=black,thick,radius=\r,inner sep=0] (1,0) circle;
\draw[fill=black,thick,radius=\r,inner sep=0] (2,0) circle;
\draw[fill=black,thick,radius=\r,inner sep=0] (0,1) circle;
\draw[fill=black,thick,radius=\r,inner sep=0] (1,1) circle;
\draw[fill=black,thick,radius=\r,inner sep=0] (2,1) circle;
\draw[fill=black,thick,radius=\r,inner sep=0] (0,2) circle;
\draw[fill=black,thick,radius=\r,inner sep=0] (1,2) circle;
\draw[fill=black,thick,radius=\r,inner sep=0] (2,2) circle;

\node at (0,1.3) {$f t$};
\node at (0.2,0) {$f$};
\node at (0.1,-1.3) {$f t^{-1}$};
\node at (1,1.3) {$h t$};
\node at (1.2,0) {$h$};
\node at (1.1,-1.3) {$h t^{-1}$};
\node at (2,1.3) {$e t$};
\node at (2.2,0) {$e$};
\node at (2.1,-1.3) {$e t^{-1}$};

\node at (1.6,2.3) {$\mathfrak{\widehat{n}_{+}}$};
\node at (0.75,0) {$\mathfrak{h}$};
\node at (0.6,-2.3) {$\mathfrak{\widehat{n}_{-}}$};

\draw[very thick] (-0.5,2.5)--(-0.5,0.6)--(1.5,0.6)--(1.5,-0.4)--(2.5,-0.4)--(2.5,2.5);

\draw[very thick] (2.5,-2.5)--(2.5,-0.6)--(0.5,-0.6)--(0.5,0.4)--(-0.5,0.4)--(-0.5,-2.5);

\end{tikzpicture}
\begin{tikzpicture}
\tikzmath{\r=0.05;};

\draw[white] (-3,-1)--(-3,1);

\draw[fill=black,thick,radius=\r,inner sep=0] (0,-2) circle;
\draw[fill=black,thick,radius=\r,inner sep=0] (1,-2) circle;
\draw[fill=black,thick,radius=\r,inner sep=0] (2,-2) circle;
\draw[fill=black,thick,radius=\r,inner sep=0] (0,-1) circle;
\draw[fill=black,thick,radius=\r,inner sep=0] (1,-1) circle;
\draw[fill=black,thick,radius=\r,inner sep=0] (2,-1) circle;
\draw[fill=black,thick,radius=\r,inner sep=0] (0,0) circle;
\draw[fill=black,thick,radius=\r,inner sep=0] (1,0) circle;
\draw[fill=black,thick,radius=\r,inner sep=0] (2,0) circle;
\draw[fill=black,thick,radius=\r,inner sep=0] (0,1) circle;
\draw[fill=black,thick,radius=\r,inner sep=0] (1,1) circle;
\draw[fill=black,thick,radius=\r,inner sep=0] (2,1) circle;
\draw[fill=black,thick,radius=\r,inner sep=0] (0,2) circle;
\draw[fill=black,thick,radius=\r,inner sep=0] (1,2) circle;
\draw[fill=black,thick,radius=\r,inner sep=0] (2,2) circle;

\node at (0,1.3) {$f t$};
\node at (0.2,0) {$f$};
\node at (0.1,-1.3) {$f t^{-1}$};
\node at (1,1.3) {$h t$};
\node at (1.2,0) {$h$};
\node at (1.1,-1.3) {$h t^{-1}$};
\node at (2,1.3) {$e t$};
\node at (2.2,0) {$e$};
\node at (2.1,-1.3) {$e t^{-1}$};

\node at (2.15,-2.4) {$\widetilde{\mathfrak{n}}_{+}$};
\node at (1.05,-2.4) {$\widetilde{\mathfrak{h}}$};
\node at (0.15,-2.4) {$\widetilde{\mathfrak{n}}_{-}$};

\draw[very thick] (1.5,2.5)--(1.5,-2.5);
\draw[very thick] (0.5,2.5)--(0.5,-2.5);

\end{tikzpicture}
}
\end{center}
\end{figure}

We would now like to extend this picture to the quantum case. This is nontrivial, since it is not obvious how the loop polarization should deform. The existence of a nice quantum deformation under which 
the loop Cartan subalgebra $\widetilde{\mathfrak{h}}$ deforms in a non-trivial way but remains commutative was an important discovery of Drinfeld. The corresponding realization of the 
quantum affine algebra is called the {\bf Drinfeld loop realization}, and we will briefly discuss it below.

\subsection{The Faddeev-Reshetikhin-Takhtajan formalism} (\cite{ES}, Ch. 11, \cite{DF}). 

\subsubsection{} To explain the essense of the loop realization of quantum affine algebras, we will consider the example $\mathfrak{g}=\mathfrak{sl}_
n$. In this case the loop realization can be interpreted in terms of 
the {\bf Faddeev-Reshetikhin-Takhtajan construction} of $U_q(\widehat{\mathfrak{sl}}_n)$. The FRT construction is the construction of $U_q(\widehat{\mathfrak{sl}}_n)$ starting from the R-matrix, which is how quantum groups appeared historically. 

Let us first describe the FRT construction and the loop generators for the positive part $U_q(\widehat{\mathfrak{sl}}_n^+)\subset U_q(\widehat{\mathfrak{sl}}_n)$, generated 
by $K_i$ and $e_i$. We will then briefly comment on a similar construction for the negative part of the quantum affine algebra, and on how the positive and negative parts are put together.

Recall that the normalized $R$-matrix for $U_q(\widehat{\mathfrak{sl}}_2)$ is
\begin{equation}
\def\arraystretch{2}
\overline{R}_{VV}(z) =
\left(
\begin{array}{cccc}
1 & 0 & 0 & 0 \\
0 & \dfrac{q(z-1)}{z-q^2} & \dfrac{1-q^2}{z-q^2} & 0 \\
0 & \dfrac{z(1-q^2)}{z-q^2} & \dfrac{q(z-1)}{z-q^2} & 0 \\
0 & 0 & 0 & 1
\end{array}
\right)
=
\left(
\begin{array}{ccc}
1 &
\begin{array}{cc}  &  \end{array} &
 \\
\begin{array}{c}  \\  \end{array} &
\fbox{B(z)} &
\begin{array}{c}  \\  \end{array} \\
 &
\begin{array}{cc}  &  \end{array} &
1
\end{array}
\right)
.
\end{equation}
For $U_q(\widehat{\mathfrak{sl}}_n)$ and $V=\mathbb{C}^n$ the matrix $\overline{R}_{VV}(z)$ is given by a similar formula: it acts by $1$ on $v_i \otimes v_i$, and by $B(z)$ on $\langle v_i\otimes v_j,\,v_j \otimes v_i\rangle$ for $j\ne i$. 

Let $\mathcal{R}$ be the universal $R$-matrix of $U_q(\widehat{\mathfrak{sl}}_n)/(\bold K-1)$, and introduce the {\bf $T$-operator}
\begin{equation}
T(z) := (\M1\otimes \pi_{V(z)}) \,(\R^{-1})\, \in \mathrm{Mat}_{n} \otimes U_q(\widehat{\mathfrak{sl}}_n)[[z^{-1}]]\end{equation}
(note that $T(\infty)=T^{(0)}$ is lower triangular). It satisfies the Faddeev-Reshetikhin-Takhtajan equation
\begin{equation}\label{frt}
\overline R^{12}(z/w)T^{13}(z)T^{23}(w) = T^{23}(w)T^{13}(z)\overline R^{12}(z/w).
\end{equation}

The Faddeev-Reshetikhin-Takhtajan construction takes \eqref{frt} as the set of defining relations of the positive half $U_q(\widehat{\mathfrak{gl}}_n^+)$ of the quantum affine algebra $U_q(\widehat{\mathfrak{gl}}_n)$, with the generators being 
the coefficients $t_{ij,k}$ of the matrix elements $t_{ij}(z)$ of $T(z)$, for $k\ge 0$ (where 
$t_{ij,0}=0$ for $i<j$ and $t_{ii}^0$ are invertible). 
More precisely, as was shown by Reshetikhin, Semenov-Tian-Shansky and Ding-Frenkel (\cite{DF}), 
the algebra defined by such generators and relations is isomorphic to 
$U_q(\widehat{\mathfrak{gl}}_n^+)$, and $U_q(\widehat{\mathfrak{sl}}_n^+)$ is obtained by taking the quotient by the relation that the appropriately defined quantum determinant $\det_q T(z)$ equals $1$.

\subsubsection{} Recall that if $A\in  \mathrm{GL}_n(R)$ for a commutative ring $R$
with upper left corner minors $D_k=\det(a_{ij}, 1\le i,j\le k)$ invertible for all $k=1,...,n$ 
then it has a unique {\bf Gauss decomposition} $A=A_-A_0A_+$, where 
$A_-$ is unipotent lower triangular, $A_0$ is diagonal, and $A_+$ is unipotent upper triangular. This also makes sense for matrices over a noncommutative ring $R$,  
when the appropriate non-commutative determinants $D_k\in R$ are invertible (see e.g. \cite{GGRW}). Thus we may consider the Gauss decomposition of the $T$-operator: 
$$
T(z)=T_-(z) T_0(z) T_+(z),
$$
\begin{equation}
\label{GaussMatForT}
T_- =
\left(
\begin{array}{cccc}
1 &   &   &  \\
* & 1 &   &  \\
* & * & 1 &  \\
* & * & * & 1\\
\end{array}
\right),
~~~
T_0 =
\left(
\begin{array}{cccc}
* &   &   &  \\
 & * &   &  \\
 &  & * &  \\
 &  &  & * \\
\end{array}
\right),
~~~
T_+ =
\left(
\begin{array}{cccc}
1 & * & * & *\\
 & 1 & * & *\\
 &  & 1 & *\\
 &  &  & 1\\
\end{array}
\right).
\end{equation} 
The Drinfeld loop realization of $U_q(\widehat{\mathfrak{gl}}_n^+)$ and $U_q(\widehat{\mathfrak{sl}}_n^+)$ is obtained from this decomposition. In particular, one can show that the matrix elements of
\begin{equation}
T_0(z) =
\left(
\begin{array}{cccc}
t^0_{00}(z) &   &   &  \\
 & t^0_{11}(z) &   &  \\
 &  & ... &  \\
 &  &  & t^0_{nn}(z) \\
\end{array}
\right)
\end{equation}
are pairwise commutative, and in the $\mathfrak{sl}_n$ case we have 
$\det_q T(z)=t^0_{00}(z)\, ...\,t^0_{nn}(z)=1$ (where $\det_q T=D_n(T)$).
 
Now, the {\bf loop generators} of $U_q(\widehat{\mathfrak{gl}}_n^+)$ are just the coefficients of the series $t_{i,i-1}^-(z)$, $t_{i,i}^0(z)$, $t_{i,i+1}^+(z)$, which are matrix elements of $T_-, T_0,T_+$ respectively.

In a similar way one can construct the loop generators of $U_q(\widehat{\mathfrak{gl}}_n^-)$, as coefficients of appropriate series $\widetilde t_{i,i-1}^-(z)$, $\widetilde t_{i,i}^0(z)$, $\widetilde t_{i,i+1}^+(z)$ derived from 
$$
\widetilde T(z):=(\pi_{V(z)}\otimes 1)(\mathcal R)
$$
(now a series with nonnegative powers of $z$). 

The loop generators satisfy some commutation relations which are the defining relations of the loop realization of $U_q(\widehat{\mathfrak{gl}}_n)$ and can be found in \cite{CP}, Chapter 12 and \cite{CP2}. We will not give them here for the sake of brevity, but will give them below for the Yangian degeneration of the quantum affine algebra. In particular, the Taylor coefficients of the series $t^0_{i,i}(z)$, $\widetilde t^0_{i,i}(z)$ commute and are the generators of the deformed loop Cartan subalgebra, discussed at the end of previous subsection. 

\subsubsection{} Now let us explain how the loop realization is used for studying finite dimensional 
representations of $U_q(\widehat{\mathfrak{sl}}_n)$. To this end, define the series $\psi_i(z) = t_{ii}^0 (z)/t_{i+1,i+1}^0 (z)$ corresponding to simple roots. Now take a finite dimensional (type I) representation $W$ of $U_q(\widehat{\mathfrak{sl}}_n)$, and pick a highest weight vector $\bold w\in W$ with respect to the usual Cartan subalgebra, of weight $\lambda$ (i.e., $W$ does not have weights higher than $\lambda$). It is clear that 
\begin{equation}
t^{+}_{ij}(z)\bold w=0,\ i<j,
\end{equation}
from weights reasons: if it is not zero, it has to have higher weight than $\lambda$. We also can choose $\bold w$ to be an eigenvector of $\psi_i(z)$:
\begin{equation}
\psi_i(z) \bold w = \Lambda_i(z^{-1}) \bold w,\ \Lambda_i\in \Bbb C[[z]], 
\end{equation}
as the coefficients of $\psi_i(z)$ are a set of commuting operators on the finite dimensional weight subspace $W[\lambda]$. Moreover, it is clear from the loop realization that if $W$ is irreducible then $\dim W[\lambda]=1$. 

\begin{theorem} (Drinfeld) There exists an irreducible finite dimensional representation $W_\Lambda$ 
of $U_q(\widehat{\mathfrak{sl}}_n)$ with highest weight $\{ \Lambda_i(z) \}$  under the loop Cartan subalgebra if and only if 
$$
\Lambda_i(z)=q^{-\deg(P_i)}\frac{P_i(q^2 z)}{P_i(z)}
$$
for some polynomials $P_i$ with constant term $1$. If such $W_\Lambda$ exists, it is unique.
\end{theorem} 

So, finite dimensional irreducible representations of $U_q(\widehat{\mathfrak{sl}}_n)$ are labeled by $(n-1)$-tuples of finite multisets in $\mathbb{C}^\times$ (the roots of $P_i$). The polynomials $P_i$ are called {\bf Drinfeld polynomials}. We have already seen them in the $\mathfrak{sl}_2$ case.

\begin{remark} (i) For any $\mathfrak{g}$ there is a Gauss decomposition of the universal $\R$-matrix of $U_q(\widehat{\mathfrak{g}})$, $\R = \R_- \R_0 \R_+$, where the second tensor components of $\R_-$, $\R_0$ and $\R_+$ are of positive, zero and negative $U_q(\mathfrak{g})$-weights, respectively. This decomposition can be used to produce a loop realization
of $U_q(\widehat{\mathfrak{g}})$. 

(ii) A similar description of finite dimensional representations is valid 
for $U_q(\widehat{\mathfrak{g}})$ for every simple Lie algebra $\mathfrak{g}$ (with $n-1$ replaced by the rank of $\mathfrak{g}$). 

(iii) We may also define $\phi_i(z)=\widetilde  t_{ii}^0 (z)/\widetilde t_{i+1,i+1}^0 (z)$, 
and the components of these series preserve $\Bbb C\bold w$, but the corresponding eigenvalues are determined by $\Lambda_i(z)$ and don't carry any new information.  

(iv) When $\mathfrak{g}=\mathfrak{sl}_2$, this classification reduces to that of Subsection \ref{classif}, and the Drinfeld polynomial $P_Y(z)$ defined in Subsection \ref{classif} 
coincides with the one defined here. 

(v) The FRT formalism naturally extends to other affine Lie algebras of classical types, 
starting with explicit expressions for their $R$-matrices in the vector representation, see 
\cite{Ba,Ji}. This gives a way to study finite dimensional representations of the corresponding quantum affine algebras. Moreover, this approach can, in fact, be extended to all affine Lie algebras (including exceptional ones), following ideas of Drinfeld, as explained in \cite{Wen}\footnote{More precisely, the paper \cite{Wen} considers the case of Yangians, but the extension to quantum affine algebras is straightforward.}; however, in this case things become less explicit.  
\end{remark} 

\subsection{Yangians} (\cite{CP}, Ch. 12). Yangians were introduced by Drinfeld in the mid 80s following previous work of V. Tarasov in the $\mathfrak{sl}_2$ case. They  are obtained from quantum affine algebras 
in the limit $\hbar\to 0$ where $q=e^{\hbar/2}$ and $z=e^{\hbar u}$.  
In this limit the trigonometric $R$-matrix with spectral parameter $\overline{R}(z)$ 
degenerates into the {\bf rational $R$-matrix with spectral parameter}, which in the $\mathfrak{sl}_2$-case looks like
\begin{equation}
\overline{R}(z)\to \dfrac{u}{u-1}\left(1-\dfrac{P}{u}\right)=
\left(
\begin{array}{cccc}
1 & & & \\
& \frac{u}{u-1} & -\frac{1}{u-1} & \\
& -\frac{1}{u-1} & \frac{u}{u-1} & \\
& & & 1
\end{array}
\right)
,~~~
P(x\otimes y) = y\otimes x.
\end{equation}
Thus the FRT equation in this limit becomes
\begin{equation}
R_Y^{12}(u-v)T^{13}(u)T^{23}(v) = T^{23}(v)T^{13}(u) R_Y^{12}(u-v)
\end{equation}
where $R_Y(u) := 1-\frac{P}{u}$ is the {\bf Yang $R$-matrix}.
This is also the case for $\mathfrak{sl}_n$ and $\mathfrak{gl}_n$. 

\begin{definition} The {\bf Yangian} $Y(\mathfrak{gl}_n)$ is the algebra generated by the entries of $n\times n$ matrices $T^{(k)},~k=0,1,2,\,...$, satisfying quadratic defining relations coming from the decomposition of the equation
\begin{equation}
R_Y^{12}(u-v)T^{1}(u)T^{2}(v) = T^{2}(v)T^{1}(u) R_Y^{12}(u-v),~~~ T(u) = 1+\sum_{n=0}^{+\infty} T^{(n)} u^{-n-1} 
\end{equation}
in $\dfrac{1}{u-v}\mathbb{C}((u^{-1},v^{-1}))$. 
\end{definition} 

\begin{exercise} Show that $Y(\mathfrak{gl}_n)$ is a Hopf algebra with coproduct given by 
\begin{equation}
\Delta T = T \otimes T,\text{ i.e. }\Delta T(u)_{ij} = \sum_{k} T(u)_{ik}\otimes T(u)_{kj}
\end{equation}
and antipode $S(T(u))=T(u)^{-1}$. 
\end{exercise} 

To get $Y(\mathfrak{sl}_n)$, one has to quotient out by the relation 
$$
\mathrm{qdet}\,T=1,
$$ 
where $\mathrm{qdet}\,T$ is an appropriate quantum determinant. 
For example,
for $n=2$ 
\begin{equation}
\mathrm{qdet}\,
\left(
\begin{array}{cc}
t_{11} & t_{12} \\
t_{21} & t_{22}
\end{array}
\right)
=
t_{11}(u)t_{22}(u+1) - t_{12}(u)t_{21}(u+1)
\end{equation}
where we decompose 
$(u+1)^{-1}=u^{-1}(1+u^{-1})^{-1} = u^{-1}-u^{-2}+u^{-3}-...$.

\begin{exercise} Show that the matrix elements of $T^{(0)}$ generate $U(\mathfrak{gl}_n)$ (namely, $T^{(0)}_{ij}=E_{ji}$).
\end{exercise} 

\subsection{The loop realization of the Yangian} (\cite{CP}, Ch. 12). 

Let us now describe the loop realization of the Yangian, for simplicity in the case $n=2$. 
To this end, as in the case of quantum affine algebras, consider the Gauss 
decomposition $T(u)=T^-(u)T^0(u)T^+(u)$. We have 
\begin{equation}
T^{+}(u) =
\left(
\begin{array}{cc}
1 & x_{+}(u) \\
0 & 1
\end{array}
\right),\
T^{0}(u) =
\left(
\begin{array}{cc}
t_{11}^0(u) & 0 \\
0 & t_{22}^0(u)
\end{array}
\right),\
T^{-}(u) =
\left(
\begin{array}{cc}
1 & 0\\
x_-(u) & 1
\end{array}
\right),
\end{equation}
with 
$$
H(u) = t_{11}^0(u)/t_{22}^0(u)=1+\sum_{n\geq 0}\, H_n u^{-n-1},~ x^{+} (u)= \sum_{n\geq 0}\, x_n^{+} u^{-n-1} ,~ x^{-}(u) = \sum_{n\geq 0}\, x_n^{-} u^{-n-1}.
$$ 
These generators satisfy the defining relations
\begin{equation}
\def\arraystretch{2}
\begin{array}{l}
\left[ H_k,H_{l} \right]=0,~~~
[H_0,x^{\pm}_{k}] = \pm 2 x^{\pm}_{k},~~~
[x_k^{+},x_{l}^{-}] = H_{k+l},
\\
\left[ H_{k+1},x_{l}^{\pm} \right] - [H_k,x_{l+1}^{\pm}] = \pm(H_k x_l^{\pm} + x_l^{\pm} H_k),
\\
\left[ x_{k+1}^\pm,x_{l}^{\pm} \right] - [x_k^\pm,x_{l+1}^{\pm}] = \pm(x_k^\pm x_l^{\pm} + x_l^\pm x_k^\pm).
\end{array}
\end{equation}

This is another presentation of the Yangian, called {\bf the loop realization}; it can be obtained 
from the corresponding realization of the quantum affine algebra in the limit $q\to 1$. 
\footnote{Thus in contrast with the quantum affine algebra, the Yangian contains only one ``half", corresponding to the negative powers of spectral parameter.}

\subsection{The Yangian of an arbitrary simple Lie algebra}  (\cite{CP}, 12.1) 
\subsubsection{} In fact, the Yangian $Y(\mathfrak{g})$ can be defined for any simple Lie algebra $\mathfrak{g}$, and we have $U(\mathfrak{g}) \subset Y(\mathfrak{g})$. The loop realization of $Y(\mathfrak{g})$ involves generators $x_i^{+}(u),x_i^{-}(u),H_{i}(u)$ 
for all simple roots $i$. We will not discuss 
this in detail here and refer the reader to \cite{CP}, Chapter 12.  

\subsubsection{}The algebra $Y(\mathfrak{g})$ is a filtered deformation (quantization) of $U(\mathfrak{g}[t])$. This means that the associated graded algebra of the Yangian (under a suitable filtration) is the universal enveloping algebra of $\mathfrak{g}[t]$:
\begin{equation}
\mathrm{gr}\, Y(\mathfrak{g}) = U(\mathfrak{g}[t] ).
\end{equation}
Thus the Yangian induces a co-Poisson-Hopf structure on $U(\mathfrak{g}[t])$, i.e., a Lie 
bialgebra structure on $\mathfrak{g}[[t]]$. It can be shown that 
this Lie bialgebra is the {\bf Yangian Lie bialgebra} which comes from the Manin triple 
$(\mathfrak{g}((t^{-1})),\mathfrak{g}[t],t^{-1}\mathfrak{g}[[t^{-1}]])$.  

There is also a Yangian version of Drinfeld's theorem on the classification of finite dimensional representations. 

\begin{theorem} There exists an irreducible finite dimensional representation $W$
of the Yangian $Y(\mathfrak{g})$ generated by a highest weight vector 
$\bold w$ with 
$$
H_i(u)\bold w=\lambda_i(u)\bold w
$$ 
if and only if 
$$
\lambda_i(u)=\frac{P_i(u+1)}{P_i(u)},\ 1\le i\le r,
$$
where $P_i$ are monic polynomials. 
\end{theorem} 

The polynomials $P_i$ are called the {\bf Drinfeld polynomials} of $W$.  

\subsection{$q$-characters} (\cite{FR})

Let us now discuss characters of finite dimensional representations of Yangians (the story for quantum affine algebras is completely parallel). This raises the question what we should mean by the character. 
One possibility is just to take the character of a $Y(\mathfrak{g})$-module 
$W$ as a $\mathfrak{g}$-module. But then we lose a lot of information ---  e.g. such characters don't know about shifts
\begin{equation}
\tau_a(T(u)) = T(u+a), 
\end{equation}
i.e., the character of $W$ and $W(a)$ is the same. We would like to have characters which determine the irreducible representation uniquely. The idea of E. Frenkel and N. Reshetikhin was to use 
the eigenvalues of the loop Cartan generators $H_i(u)$ to define a more refined notion of character. 

For simplicity consider the case $\mathfrak{g}=\mathfrak{sl}_2$. In this case the 
loop Cartan subalgebra is generated by the coefficients of the series $H(u)$. 

\begin{proposition} 
All eigenvalues of $H(u)$ on $W$ 
are of the form $\dfrac{P(u+\tfrac{1}{2})Q(u-\tfrac{1}{2})}{P(u-\tfrac{1}{2})Q(u+\tfrac{1}{2})}$ for uniquely determined coprime monic polynomials $P,Q$.
\end{proposition} 

Thus, to every such eigenvalue we can attach a Laurent monomial in variables $Y_b$, $b\in \Bbb C$, as follows: 
\begin{equation}
\dfrac{P(u)}{Q(u)} = \dfrac{\prod\limits_{i=1}^{n} (u-a_i)}{\prod\limits_{j=1}^{m} (u-b_j)}
~~~ \mapsto ~~~
Y_{a_1}\, ... ~ Y_{a_n} \, Y_{b_1}^{-1}\, ... ~ Y_{b_m}^{-1}.
\end{equation}
Note that for the highest weight vector $Q=1$ and $P(u)=P_W(u+\tfrac{1}{2})$ where $P_W$ is the Drinfeld polynomial of $W$, so 
the monomial attached to the highest weight vector involves only positive powers of $Y_b$. 

\begin{definition} The {\bf $q$-character} $\chi_W$ of a finite-dimensional representation $W$ of $Y(\mathfrak{sl}_2)$ is the sum of the monomials attached to all eigenvalues of $H(u)$ on $W$ (taken with multiplicities). 
\end{definition} 

Let $K_0(\Rep\ Y({\mathfrak{sl}}_2))$ be the Grothendieck ring of the category 
of finite dimensional representations of $Y({\mathfrak{sl}}_2)$. 

\begin{theorem} There exists an injective ring homomorphism
\begin{equation}
\chi: K_0(\Rep\, Y({\mathfrak{sl}}_2))
~~~ \to ~~~
\mathbb{Z}[Y_a^{\pm 1} \, ; \, a\in \mathbb{C}] 
\end{equation}
given by $W\mapsto \chi_W$. 
\end{theorem} 

In particular, we see that 
\begin{equation}
\chi_{W\otimes U}=\chi_W\chi_U. 
\end{equation}

\begin{corollary} 
$K_0(\Rep\ Y({\mathfrak{sl}}_2))$ is an integral domain.
\end{corollary} 

\begin{example} Let $V = V_1$ be the 2-dimensional irreducible representation of $\mathfrak{sl}_2$, and $V_1(a)$ be the corresponding evaluation representation of $Y(\mathfrak{sl}_2)$, $a\in \mathbb{C}$. 
Then we have two eigenvectors of $H(u)$ -- $v_+$ (of weight $1$) and $v_-$ (of weight $-1$). 
One can compute that the corresponding eigenvalues of $H(u)$ give the following rational functions 
$P/Q$: 
\begin{equation}
\dfrac{P_{v_{+}}(u)}{Q_{v_{+}}(u)} = u-a+\tfrac{1}{2},
~~~
\dfrac{P_{v_{-}}(u)}{Q_{v_{-}}(u)} = \dfrac{1}{u-a-\tfrac{1}{2}}.
\end{equation}
This implies that 
$$
\chi_{V_1(a)} = Y_{a-\tfrac{1}{2}} + Y_{a+\tfrac{1}{2}}^{-1}.
$$
Hence 
\begin{equation}
\chi_{V_1(a)\otimes V_1(b)} = \left(Y_{a-\tfrac{1}{2}}  + Y_{a+\tfrac{1}{2}}^{-1}\right)\left(Y_{b-\tfrac{1}{2}} + Y_{b+\tfrac{1}{2}}^{-1}\right) = Y_{a-\tfrac{1}{2}} Y_{b-\tfrac{1}{2}}+ Y_{a-\tfrac{1}{2}} Y_{b+\tfrac{1}{2}}^{-1} + Y_{a+\tfrac{1}{2}}^{-1} Y_{b-\tfrac{1}{2}} + Y_{a+\tfrac{1}{2}}^{-1} Y_{a+\tfrac{1}{2}}^{-1}.
\end{equation}

It follows (without any computation!) that this tensor product is irreducible unless $a-b=\pm 1$ -- as otherwise there is only one monomial which has only positive powers of the variables -- the one corresponding to the highest weight. On the other hand, if $a=b\pm 1$ then there is also monomial $1$, which is the $q$-character of the trivial representation. Of course, in this example it is easy to see that this tensor product is irreducible by direct computation, but this method can be used in more complicated situations as well and provides a simple test for irreducibility of finite dimensional representations of Yangians and quantum affine algebras, in particular of representations constructed as tensor products.  

Also, we get from this the $q$-character of the 3-dimensional irreducible representation $V_2$. Namely, we have a short exact sequence  
\begin{equation}
0
~ \to ~
\mathbb{C}
~ \to ~
V_1(-\tfrac{1}{2})\otimes V_1(\tfrac{1}{2})
~ \twoheadrightarrow ~
V_2(0)
~ \to ~
0
\end{equation}
so
\begin{equation}
\chi_{V_2(0)} =
Y_{-1} Y_{0} + Y_{-1} Y_{1}^{-1} +  Y_{0}^{-1} Y_{1}^{-1}.
\end{equation}
\end{example} 

\begin{exercise} Show that
\begin{equation}
\chi_{V_{m}(0)} = Y_{-{m\over 2}}...Y_{m-4\over 2} Y_{m-2\over 2} +Y_{-{m\over 2}}...Y_{m-4\over 2}Y_{m\over 2}^{-1} +...+ Y_{-{m-2\over 2}}^{-1}...Y_{m-2\over 2}^{-1}Y_{m\over 2}^{-1}.
\end{equation}
\end{exercise} 

Using $q$-characters, one can prove that the Grothendieck ring of finite dimensional 
representations of $Y(\mathfrak{g})$ is the polynomial ring in affinizations of fundamental representations of $\mathfrak{g}$ with arbitrary shifts:

\begin{theorem} $K_0(\Rep\, Y(\mathfrak{g})) \cong \mathbb{C}[\widehat V_{\omega_i}(a) ~ ; a\in \mathbb{C}, i=1,...,r]$, 
where $\widehat V_{\omega_i}$ are affinizations of the fundamental representations of $\mathfrak{g}$, i.e. irreducible representations with Drinfeld polynomials $P_i=u$ and $P_j=1$ for $j\ne i$.  
\end{theorem}

\vskip .1in

{\bf P.E.}: Department of Mathematics, MIT, Cambridge, MA 02139, USA, 
\email{etingof@math.mit.edu}

{\bf M.S.}: Perimeter Institute for Theoretical Physics, Waterloo, ON N2L 2Y5, Canada, \linebreak \email{semenyakinms@gmail.com}
\end{document}